\documentclass[12pt]{amsart}
\usepackage{amsmath}
\usepackage{amscd}
\usepackage{amssymb}
\usepackage{amsfonts}
\usepackage{amsthm}
\usepackage{bbm}
\usepackage{cancel}
\usepackage{color}
\usepackage{eucal}
\usepackage{enumerate,yfonts}
 
 \usepackage{fullpage}
\usepackage{pdfsync}
 \usepackage[all,cmtip]{xy}
\usepackage{graphicx}
\usepackage{graphics}
\usepackage{hyperref}
\usepackage{latexsym}
\usepackage{mathrsfs}
 \usepackage{placeins}
\usepackage{pstricks}
\usepackage{bookmark}

\usepackage{stmaryrd}
\usepackage{url}

\newtheorem{thm}{Theorem}[section]
\newtheorem{corollary}[thm]{Corollary}
\newtheorem{lemma}[thm]{Lemma}

\newtheorem{proposition}[thm]{Proposition}
\newtheorem{prop}[thm]{Proposition}
\newtheorem{conjecture}[thm]{Conjecture}
\newtheorem{thm-dfn}[thm]{Theorem-Definition}

\theoremstyle{definition}
\newtheorem{definition}[thm]{Definition}
\newtheorem{remark}[thm]{Remark}

\numberwithin{equation}{section}

\newcommand{\fg}{{\mathfrak g}}
\newcommand{\ft}{{\mathfrak t}}

\newcommand{\fb}{{\mathfrak b}}

\newcommand{\fn}{{\mathfrak n}}

\newcommand{\rS}{{\mathrm S}}

\newcommand{\rW}{{\mathrm W}}
\newcommand{\rU}{{\mathrm U}}

\newcommand{\bC}{{\mathbb C}}

\newcommand{\bG}{{\mathbb G}}
\newcommand{\bZ}{{\mathbb Z}}

\newcommand{\mD}{\mathcal{D}}

\newcommand{\mE}{\mathcal{E}}
\newcommand{\mF}{\mathcal{F}}

\newcommand{\mA}{\mathcal{A}}
\newcommand{\mM}{\mathcal{M}}

\newcommand{\mO}{\mathcal{O}}
\newcommand{\mL}{\mathcal{L}}
\newcommand{\mH}{\mathcal{H}}

\newcommand{\mG}{\mathcal{G}}

\newcommand{\mB}{\mathcal{B}}

\newcommand{\calF}{{\mathcal F}}

\newcommand{\cO}{{\mathcal O}}
\newcommand{\cA}{{\mathcal A}}
\newcommand{\cF}{{\mathcal F}}
\newcommand{\cN}{{\mathcal N}}

\newcommand{\cE}{{\mathcal E}}

\newcommand{\cZ}{{\mathcal Z}}
\newcommand{\cS}{{\mathcal{S}}}
\newcommand{\cM}{{\mathcal{M}}}
\newcommand{\cG}{{\mathcal{G}}}

\newcommand{\on}{\operatorname}

\newcommand{\tU}{\widetilde U}

\newcommand{\ra}{\rightarrow}
\newcommand{\la}{\leftarrow}

\newcommand{\is}{\simeq}

\newcommand{\Loc}{\on{LocSys}}

\newcommand{\nc}{\newcommand}

\nc{\al}{{\alpha}} \nc{\be}{{\beta}} \nc{\ga}{{\gamma}}
\nc{\ve}{{\varepsilon}} \nc{\Ga}{{\Gamma}} 
\nc{\La}{{\Lambda}}

\nc{\ad }{{\on{ad }}}

\nc{\aff}{{\on{aff}}} \nc{\Aff}{{\mathbf{Aff}}}

\nc{\der}{{\on{der}}}

\nc{\diag}{{\on{diag}}}

\nc{\Fl}{{\calF\ell}}

\nc{\Hg}{{\on{Higgs}}}
\newcommand{\Hom}{{\on{Hom}}}

\nc{\Id}{{\on{Id}}}

\nc{\Ind}{{\on{Ind}}}

\nc{\Op}{{\on{Op}}}

\newcommand{\pr}{{\on{pr}}}
\newcommand{\Res}{{\on{Res}}}
\nc{\res}{{\on{res}}}

\newcommand{\Spec}{{\on{Spec}}}

\nc{\tr}{{\on{tr}}}

\nc{\GSp}{{\on{GSp}}} \nc{\GU}{{\on{GU}}} \nc{\SL}{{\on{SL}}}
\nc{\SU}{{\on{SU}}} \nc{\SO}{{\on{SO}}}
\newcommand{\Ad}{{\on{Ad}}}

\nc{\nh}{{\Loc_{J^p}(\tau')}}
\nc{\bnh}{{\Loc_{\breve J^p}(\tau')}}

\nc{\bU}{{\overline{U}}} 
\nc{\IC}{{\on{IC}}}

\newcommand{\br}{\begin{rouge}}
\newcommand{\er}{\end{rouge}}

\newcommand{\bb}{\begin{bluet}}
\newcommand{\eb}{\end{bluet}}

\newcommand{\ul}{{\underline\lambda}}

\newcommand{\prolim}{\textup{}\varprojlim\textup{}}

\nc{\ot}{\otimes}

\nc{\oh}{{\operatorname{H}}}
\nc{\gr}{{\operatorname{gr}}}
\nc{\rk}{{\operatorname{rank}}}
\nc{\codim}{{\operatorname{codim}}}
\nc{\img}{{\operatorname{Im}}}
\nc{\Span}{{\operatorname{Span}}}
\nc{\Img}{\operatorname{Im}}

\newcommand{\beqn}{\begin{equation*}}
\newcommand{\eeqn}{\end{equation*}}

\newcommand{\beq}{\begin{equation}}
\newcommand{\eeq}{\end{equation}}

\newcommand{\bern}{\begin{eqnarray*}}
\newcommand{\eern}{\end{eqnarray*}}
\nc{\Fano}[1]{\on{Fano}_{#1}}

\newcommand{\lL}{\lambda+\Lambda}
\newcommand{\hlL}{\widehat{\lambda+\Lambda}}

\newcommand{\quash}[1]{}  

\begin{document}
\title{non-linear Fourier transforms and the Braverman-Kazhdan conjecture}
        \author{Tsao-Hsien Chen}
        \address{Department of Mathematics, University of Chicago, Chicago, IL 60637, USA.
          }
        \email{chenth@math.uchicago.edu}
        \thanks{Tsao-Hsien Chen was supported in part by the AMS-Simons travel grant.}

\begin{abstract}
In this article we
prove a conjecture of Braverman-Kazhdan in \cite{BK} on 
the acyclicity of gamma sheaves in the de Rham setting. 
The proof 
relies on the techniques developed in \cite{BFO} on
Drinfeld center of Harish-Chandra bimodules and character $D$-modules. 
As an application, we
show that the functors of convolution with gamma $D$-modules, which can be 
viewed as a version of 
\emph{non-linear Fourier transforms}, 
commute with induction functors and 
are exact on the category of 
admissible $D$-modules on a reductive group.
\end{abstract}

\maketitle

\setcounter{tocdepth}{1} 

\section{Introduction}
Let $k$ be an algebraic closure of 
a finite field.
The Fourier-Deligne transform 
\[\mathrm F_V:D^b_c(V,\bar{\mathbb Q}_\ell)\ra D^b_c(V,\bar{\mathbb Q}_\ell)\]
on the derived category of $\ell$-adic sheaves 
on a vector space $V$ over $k$ had found remarkable applications to 
number theory and representation theory.
The Fourier-Deligne transform has the following remarkable properties:
(1) $\mathrm F_V$ is exact with respect to the perverse $t$-structure 
on $D^b_c(V,\bar{\mathbb Q}_\ell)$ (see \cite{KL}) and (2) when 
$V=\fg$ is a reductive Lie algebra, the functor $\mathrm F_V$ commutes 
with induction functors (see \cite{L}).  

Let $G$ be a reductive group over $k$ and $\breve G$ be the 
dual group of $G$ over $\bC$.
In their work \cite{BK,BK1}, Braverman and Kazhdan 
associated to each representation $\rho:\breve G\ra\on{GL}(V_\rho)$
of the dual group $\breve G$, satisfying some mild technical conditions, a perverse sheaf 
$\Psi_{G,\rho}$ on $G$ called \emph{gamma sheaf} and study the functor 
\[\mathrm F_{G,\rho}:=(-)*\Psi_{G,\rho}:D^b_c(G,\bar{\mathbb Q}_\ell)\ra D^b_c(G,\bar{\mathbb Q}_\ell)\]
of convolution with $\Psi_{G,\rho}$. The functor $\mathrm F_{G,\rho}$ can be thought as 
a non-linear analogue of the Fourier-Deligne transform and 
they conjectured the following properties parallel to the properties
(1) and (2) above:
\begin{conjecture}\label{ind}
(1) 
$\mathrm F_{G,\rho}$ is exact with respect to the 
perverse $t$-structure.
(2)
$\mathrm F_{G,\rho}$ commutes with induction functors.
\end{conjecture} 

It is shown in \emph{loc. cit.} that 
the property (2) of the conjecture above follows from the following 
acyclicity of gamma sheaves:
\begin{conjecture}[Conjecture 9.2 \cite{BK}]\label{acyc}
Let $B$ be a Borel subgroup and consider 
the quotient map $\pi_U:G\ra G/U$, 
where $U$ is the unipotent radical of $B$. 
Then $(\pi_U)_!\Psi_{G,\rho}$ is supported on 
$T=B/U\subset G/U$. In other words, for any $g\in G-B$ we have 
$H^*_c(gU,i^*\Psi_{G,\rho})=0$. Here $i:gU\ra G$ denotes the 
inclusion map.

\end{conjecture}
In \cite{CN}, Cheng and Ng\^o established Conjecture \ref{acyc} for
$G=GL_n$ by generalizing the argument in \cite{BK} for $GL_2$.

The gamma sheaf $\Psi_{G,\rho}$
has an obvious analogue in the $D$-modules setting, which we call it gamma $D$-module, and 
the goal of this paper is to study 
$D$-modules analogue of Conjecture \ref{ind} and Conjecture \ref{acyc}. 

We now state our main results. 
Fix a Borel subgroup $B$ and a maximal tours $T\subset B$.
We begin with a construction, due to Braverman and Kazhdan, of 
gamma $D$-module 
$\Psi_{G,\ul}$
attached to a collection 
of co-characters $\ul=\{\lambda_1,...,\lambda_r\}$ of the maximal torus $T$.
The relation between $\Psi_{G,\ul}$ and the gamma 
$D$-module attached to a representation of the dual group 
$\breve G$ will be explained in \S\ref{cons of gamma d mod}.
Let $\ul=\{\lambda_1,...,\lambda_r\}$ be a collection of 
co-characters of $T$. Consider the following maps $\pr_\ul:=\prod\lambda_i:\bG_m^r\ra T$, $\tr:\bG_m^r\ra\bG_a, (x_1,...,x_r)\ra\sum x_i$. 
Assume $\ul$ is stable under the action of the Weyl group $\rW$ and
each $\lambda_i\in\ul$ is $\sigma$-positive (see Definition \ref{sigma positive}), then Braverman and Kazhdan showed that
\[\Psi_{\ul}:=(\pr_\ul)_*\tr^*(\bC[x]e^x),\]
where $\bC[x]e^x$ is the exponential $D$-module on $\bG_a=\on{Spec}(\bC[x])$,
is a (de Rham) local system on the image of $\pr_\ul$ equipped with a natural 
$\rW$-equivariant structure. 
Assume further that $\pr_\ul$ is onto. Then $\Psi_{\ul}$ is a 
$\rW$-equivariant local system on $T$. Moreover, 
the $\rW$-equivariant structure on $\Psi_{\ul}$ induces a 
$\rW$-action on the induction $\Ind_{T\subset B}^G(\Psi_{\ul})$ and 
the gamma $D$-module $\Psi_{G,\ul}$ is defined as 
the $\rW$-invariant factor of $\Ind_{T\subset B}^G(\Psi_{\ul})$:
\[\Psi_{G,\ul}:=\Ind_{T\subset B}^G(\Psi_{\ul})^\rW.\]
We prove the following equivalent form of Conjecture \ref{acyc}
in the $D$-module setting: 
\begin{thm}(see Theorem \ref{main thm})
$\on{Av}_U(\Psi_{G,\ul}):=(\pi_U)_*\Psi_{G,\ul}$ is supported on $T=B/U\subset G/U$. 

\quash{
(2) The functor of convolution with gamma $D$-module
\[\mathrm F_{G,\ul}:=(-)*\Psi_{G,\ul}: 
D(G)_{hol}\ra D(G)_{hol}\] commutes with induction functor.
Here $D(G)_{hol}$ denotes the derived category of holonomic $D$-modules on $G$.}
\end{thm}

In view of exactness property in
Conjecture \ref{ind}, we establish the following result.
We call a holonomic $D$-module on $G$ \emph{admissible} if
the action of the center $Z$ of the 
universal enveloping algebra $U(\fg)$, viewing as invariant differential operators, is locally finite. We denote by $\cA(G)$ the abelian category of admissible $D$-modules on $G$ and $D(\cA(G))$ be the corresponding derived category. Denote by $D(G)_{hol}$ the derived category of 
holonomic $D$-modules on $G$.
Consider the functor $\mathrm F_{G,\ul}:=(-)*\Psi_{G,\ul}:D(G)_{hol}\ra D(G)_{hol}$
of convolution with gamma $D$-module $\Psi_{G,\ul}$.
\begin{thm}(see Theorem \ref{exact of FT})
The functor $\mathrm F_{G,\ul}$ restricts to a functor 
\[\mathrm F_{G,\ul}:D(\mA(G))\ra D(\mA(G))\]
which is exact with respect to the natural $t$-structure. That is, we have 
$\mathrm F_{G,\ul}(\cM)\in\cA(G)$ for $\cM\in\cA(G)$.

\end{thm}

The proofs of Theorem 1.3 and Theorem 1.4 make use of
certain remarkable 
character $D$-module $\cM_\theta$ on $G$ with 
generalized central character $\theta\in\breve T/\rW$ and
the results in \cite{BFO} on
the equivalence between 
Drinfeld center of Harish-Chandra bimodules 
and character $D$-modules. 
In more details, 
we construct for each $\theta\in\breve T/\rW$ 
a $\rW$-equivariant 
local system $\mE_\theta$ 
on $T$ and consider the character $D$-module  
$\cM_\theta:=\Ind_{T\subset B}^G(\mE_\theta)^\rW$
(see 
\S\ref{central Loc},
\S\ref{CS M_theta})\footnote{The author learned the 
existence of $\mE_\theta$ from R.Bezrukavnikov.
}. 
Using the results in \cite{BFO}, we 
prove the acyclicity of $\cM_\theta$ similar to Conjecture \ref{acyc} (see Theorem \ref{Key})
and compute the convolution of $\cM_\theta$ with the gamma $D$-module $\Psi_{G,\ul}$
(see Theorem \ref{conv with M_theta}). 
A key step in the proof of the acyclicity of $\cM_\theta$ is the identification of 
the global section of $\mE_\theta$ with certain element in the Drinfeld center of 
Harish-Chandra bimodules (see \S\ref{center}).
Those results together with some simple vanishing lemmas (see \S\ref{vanishing lemmas}) imply
Theorem 1.3.
Theorem 1.4 follows from a computation of the convolution of
$\Psi_{\ul}$ with certain (pro-) local system 
on 
$T$ (see Lemma \ref{Psi_c}) and the exactness property of the
(twisted) Harish-Chandra functor in
\cite[Corollary 3.4]{BFO} (see also \cite{CY}).

It seems that our methods may be applicable to
the setting of 
$\ell$-adic sheaves. 
A new ingredient needed is an appropriate version of the   
results in \cite{BFO} in the $\ell$-adic setting.

\quash{
The paper is organized as follows. In Section~\ref{sec-pre} we 
recall some facts about symmetric pairs and  
introduce a class of 
representations of equivariant fundamental groups.
In Section~\ref{sec-parabolic} we study parabolic induction functors for certain $\theta$-stable 
parabolic subgroups. In Section~\ref{sec-FT}, we prove Theorem \ref{main theorem}:
the Fourier transform defines a bijection between the set of nilpotent 
orbital complexes and the class of representations of equivariant fundamental groups 
introduced in Section~\ref{sec-pre}.  In Section~\ref{sec-Hess}
and Section~\ref{Rep}, we discuss applications of our results to 
cohomology of Hessenberg varieties and representations of Hecke algebras of symmetric groups  
at $q=-1$.
Finally, in Section~\ref{sec-conj}, we propose 
a conjecture that gives a more precise description of the bijection in Theorem \ref{main theorem}.}

The paper is organized as follows. In Section 2 we recall some facts about 
algebraic groups and $D$-modules. In Section 3 we introduce 
the character $D$-module $\cM_\theta$ and prove the 
acyclicity of it using \cite{BFO}. In Section 4 we recall the construction 
of gamma $D$-modules $\Psi_{G,\ul}$ and compute the convolution of 
$\Psi_{G,\ul}$ with $\cM_\theta$. In Section 5 we prove 
Theorem \ref{main thm}: the acyclicity of 
gamma $D$-module $\Psi_{G,\ul}$. In Section 6 we consider the 
functor $\mathrm F_{G,\ul}:=(-)*\Psi_{G,\ul}$ of convolution with the
gamma $D$-module and we prove that $\mathrm F_{G,\ul}$
commutes with induction functors (see Theorem \ref{commutes with ind}) and 
is exact on the category of admissible $D$-modules on $G$ (see Theorem \ref{exact of FT}).

{\bf Acknowledgement.}
The author would like to thank 
R.Bezrukavnikov and Z.Yun for useful discussions.
He also thanks the Institute 
of Mathematics Academia Sinica in Taipei for support, hospitality, and a nice research environment.

\section{Notations}
\subsection{Group data}
Let $G$ be a reductive group over $\bC$. Let $B\subset G$ be a 
Borel subgroup and $T\subset B$ be a maximal torus. Let 
by $\Lambda$ be the weight lattice. Let $\Phi$ be the root system determined by 
$(G,T)$ and 
$\Phi^+$ be the set of positive roots determined by 
$(G,B)$ and $\Pi$ be the set of simple roots. We denote by $\on{W}$ be 
the Weyl group and $\rW_{a}=\rW\rtimes\Lambda$ the affine Weyl group.

We denote by $\breve G$ the dual group of $G$ and 
$\breve T$ the dual maximal torus. 
We fix a non-degenerate $\rW$-invariant form $(,)$ on $\ft$ and use it to
identify $\ft^*$ with $\breve\ft$.

We denote by 
$\mB=G/B$ the flag variety, 
$X=G/N$ the basic affine space, and 
$Y=X\times X$.

We denote by $\fg$, $\fb$, $\ft$ (resp. 
$\breve\fg$, $\breve\fb$, $\breve\ft$) the Lie algebras of $G,B,T$ (resp. 
$\breve G,\breve B,\breve T$).
For any $\xi\in\breve T$ (resp. $\mu\in\breve\ft$) we write 
 $[\xi]\in\breve T/\rW$ (resp. $[\mu]\in\breve T=\breve\ft/\Lambda$) for its image in $\breve T/\rW$ (resp. $\breve T$).

We denote by $G_{rs}$ (resp. $T_{rs}$) be the open subset consisting of regular semi-simple elements
in $G$ (resp. $T$). 

\subsection{$D$-modules}
For any smooth variety $X$ over $\bC$ we denote by $\mM(X)$ the abelian category of 
$D$-modules on $X$ and $\mM(X)_{hol}$ the full subcategory of holonomic $D$-modules. 
We write 
$D(X)$ for the bounded derived category $D$-modules on $X$
 and $D(X)_{hol}$ for the bounded derived category of holonomic $D$-modules on $X$.
We denote by $\mO_X$ and 
$\mD_X$ the sheaf of functors on $X$ and 
the sheaf of differential operators on $X$ respectively.
For a $\cF\in D(X)$, we denote by $\mH^i(\cF)\in\cM(X)$ its $i$-th cohomology $D$-module.

Let $f:X\ra Y$ be a map between smooth varieties. 
Then we have functors 
$f_*,f^!$ between $D(X)$ and $D(Y)$ and 
functors $f^*,f^!,f_*,f_!$ between $D(X)_{hol}$ and $D(Y)_{hol}$.
Note all functors above are understood in the derived sense. 
We denote by 
$\mathbb D$ the duality functor on 
$D(X)_{hol}$.
We define $f^0:=f^![\dim Y-\dim X]$.
When $Y=\on{Spec}(\bC)$ is a point 
we sometimes write $R\Gamma_{\text{dR}}(\cM):=f_*(\cM)$ and 
$H^i_{\text{dR}}(\cM):=\mH^i(f_*(\cM))$.

For $\cM,\cM'\in D(X)_{h}$, we define 
$\cM\otimes\cM=\Delta^*(\cM\boxtimes\cM')$
and $\cM\otimes^!\cM=\Delta^!(\cM\boxtimes\cM')$
where $\Delta:X\ra X\times X$ is the diagonal embedding.

For a $D$-module $\cM$ on $X$ we denote by
$\Gamma(\cM)$ (resp. $R\Gamma(\cM)$) the global sections (resp.
derived global sections)
of $\cM$ regrading 
as quasi-coherent $\mO_X$-module.

By a local system on $X$ we mean 
a $\cO_X$-coherent $D$-module on $X$, a.k.a. 
a vector bundle on $X$ with a flat connection. 

Assume $f:X\ra Y$ is a principal $T$-bundle. 
A $D$-module $\cF$ on $X$ is called $T$-monodromic
if it is weakly $T$-equivariant (see \cite[Section 2.5]{BB1}). 
We denote by $\cM(X)_{mon}$ the category consisting of 
$T$-monodromic $D$-modules on $X$.
A object $\cF\in D(X)$ is called $T$-monodromic if 
$\mH^i(\cF)\in\cM(X)_{mon}$ for all $i$. We denote by
$D(X)_{mon}$ the full subcategory consisting of 
$T$-monodromic objects.
Let $\cF\in\cM(X)_{mon}$.
For any $\mu\in\breve\ft\ (\is\ft^*)$, we denote $\Gamma^{\hat\mu}(\cF)$ 
(resp. $R\Gamma^{\hat\mu}(\cF)$) the maximal summand 
of $\Gamma(X,\cF)$ (resp. $R\Gamma(X,\cF)$)
where $U(\ft)$ (acting as infinitesimal
translations along the action of $T$) acts with the generalized eigenvalue $\mu$.

\quash{
The action induces a
map $U(\ft)\ra\Gamma(X,\mD_X)$, hence for any 
$\cF\in\cM(X)$ and $U\subset X$ open, the section 
$\Gamma(U,\cF)$ is 
naturally a $U(\ft)$-module.
A $D$-module $\cF$ on $X$ is called $T$-monodromic if the action of 
$U(\ft)$ on $\Gamma(U,\cF)$ is locally finite for all open $U\subset X$.}

\section{Drinfeld center and Character $D$-modules $\cM_\theta$}\label{CS}

In this section 
we attach to 
each  
$\theta\in\breve T/\rW$ a $\rW$-equivariant local system 
$\mE_\theta$
on $T$ and use it to construct a character $D$-module $\mM_{\theta}$ on $G$. 
The main result of this section is an acyclicity property of $\mM_\theta$ (see Theorem \ref{Key}).
The proof uses the results of Drinfeld center of 
Harish-Chandra bimodules 
and character $D$-modules in
\cite{BFO}.

\subsection{Functors between equivariant categories}
Let $H$ be a smooth algebraic group acting on a smooth variety $Z$. 
We denote by $\mM_H(Z)$ (resp. $\mM_H(Z)_{hol}$) the category of $H$-equivaraint 
$D$-modules (resp. holonomic $D$-modules) on $Z$.
We denote by $D_H(Z)$ the $H$-equivariant derived category of 
$D$-modules on $Z$ and 
$D_H(Z)_{hol}$ the $H$-equivariant derived category of holonomic 
$D$-modules on $Z$. 

Let $f:Z\ra Z'$ be be a map between two smooth varieties. 
Assume $H$ acts on $Z$ and $Z'$ and $f$ is compatible with those $H$-actions.
Then the functors $f^*,f^!,f_*,f_!$ lift to functors between 
$D_H(Z)$ and $D_H(Z')$.

For any closed subgroup $H'\subset H$ 
the forgetful functor 
$\on{oblv}_{H'}^H:D_H(Z)\ra D_{H'}(Z)$ admits a
right adjoint 
\[\on{Ind}_{H'}^H:D_{H'}(Z)\ra D_{H}(Z)\]

Consider the quotient map $\pi_U:G\ra X=G/U$. It induces 
functors 
\beq\label{Av!}
\on{Av}_U:D_G(G)\stackrel{\on{oblv}_B^G}\ra D_B(G)\stackrel{(\pi_U)_*}\ra D_B(X)
\eeq
\[\on{Av}_{U!}:D_G(G)_{hol}\stackrel{\on{oblv}_B^G}\ra D_B(G)_{hol}\stackrel{(\pi_U)_!}\ra D_B(X)_{hol}\]
between equivaraint derived categories. Here $G$ acts on $G$ by 
the conjugation action. 
We call $\on{Av}_U$ (resp. $\on{Av}_{U!}$)
star averaging functor (resp. shriek averaging functor).
The functor $\on{Av}_U$ admits a right adjoint 
\[\on{Av}_{G}:=\on{Ind}_B^G\circ\pi_U^*:D_B(X)\ra D_B(G)\ra D_G(G).\]

We shall 
recall Lusztig's  induction and restriction functors.
Consider \[T=B/U\stackrel{r}\la B\stackrel{u}\ra G\]
We define 
\beq\label{Ind}
\Ind_{T\subset B}^G:=\Ind_B^G\circ u_*\circ r^!:D_T(T)\ra D_G(G),\ \ \ \Res_{T\subset B}^G:=r_*\circ u^!:D_G(G)\ra D_T(T).
\eeq

Here is an equivalent definition of $\Ind_{T\subset B}^G$:
consider the Grothendieck-Springer simultaneous resolution:
\beq\label{GS map}
\xymatrix{
\widetilde G\ar[d]^{\tilde q}\ar[r]^{\tilde c}&T\ar[d]^q&\\
G\ar[r]^c& T/\rW}
\eeq
where $\widetilde G$ consists of pairs $(g,hB)\in G\times G/B$ such 
that $h^{-1}gh\in B$, the map $\tilde c$ is given by $\tilde c(x,hB)=h^{-1}gh\on{mod} U$,
and $\tilde q, q$ are the natural projection maps. The group $G$ acts on
$\widetilde G$ by the formula $x(g,hB)=(xgx^{-1},xhB)$ and $\tilde q$ (resp. $\tilde c$)
is $G$-equivariant where $G$ acts on $G$ (resp. $T$) via the 
conjugation action (resp. trivial action).
We have 
\beq\label{Ind GS}
\Ind_{T\subset B}^G\is\tilde q_*\tilde c^*:D_T(T)\ra D_G(G).
\eeq

Consider the following maps \[G\stackrel{p}\la G\times G/B\stackrel{q}\ra Y/T=(G/U\times G/U)/T\]
where $p(g,xB)=g$ and $q(g,xB)=(gxU,xU)\on{mod} T$. 
The group $G$ acts on $G$, $G\times \mB$ and $Y/T$ by the formulas 
$a\cdot g=aga^{-1}$, $a\cdot(g,xU)=(aga^{-1},axB)$, $a(xU,yU)=(axU,ayU)$.
One can check that $p$ and $q$ are compatible with those $G$-actions.

Following \cite{MV}, we consider the functor 
\beq\label{HC}
\mathrm{HC}=q_*p^![-\dim G/B]:
D(G)\ra
D(Y/T).
\eeq The functor above admits a right adjoint 
$\mathrm{CH}=p_*q^*[\dim G/B]:D(Y/T)\ra D(G)$. 
We use the same notations for the corresponding functors between
$G$-equivariant derived categories
$D_G(G)$ and $D_G(Y/T)$.
Following \cite{G}, we call
$\mathrm{HC}$ the Harish-Chandra functor.

Recall the following well-known fact:
\begin{lemma}[Theorem 3.6 \cite{MV}]\label{CHHC}
\begin{enumerate}
\item
Let $\sigma:\widetilde\cN\ra\cN$ be the Springer resolution of the nilpotent cone $\cN$
and let $Sp:=\sigma_!\mO_{\widetilde\cN}$ be the Springer $D$-module. 
For any $\mF\in D(G)$
there is canonical isomorphism 
\[\mathrm{CH}\circ\mathrm{HC}(\mF)\is\mF*Sp.\]

\item
We have a canonical isomorphism 
$\mathrm{CH}\circ\mathrm{HC}\is\on{Av}_G\circ\on{Av}_U[-]$.

\item
The identity functor is a direct summand of 
$\mathrm{CH}\circ\mathrm{HC}\is\on{Av}_G\circ\on{Av}_U[-]$.

\end{enumerate}

\end{lemma}
We will need the following properties of induction functors.
\begin{prop}[Theorem 2.5 and Proposition 2.9 in \cite{BK1}]\label{properties of ind}
\begin{enumerate}
\item For any local system $\mF$ on $T$, we have 
$\on{Ind}_{T\subset B}^G(\mF)\in\mM_G(G)$. 

\item 
Let $\mF$ be a local system on $T$.
For any $w\in\rW$, there is a
canonical isomorphism 
\[\on{Ind}^G_{T\subset B}(\mF)\is\on{Ind}_{T\subset B}^G(w^*\mF).\]
\item
Let $W'\subset W$ be a subgroup and 
$\mF$ be a $\rW'$-equivariant local system on $T$. There is a canonical $\rW'$-action 
on $\Ind_{T\subset B}^G(\mF)$.
\item
Let 
$\mF\in D(T)$ and 
$\mG\in D_G(G)$. Assume $\mF':=\on{Av}_U(\mG)$ is supported on $T=B/U$.
There is an isomorphism 
\[
\Ind_{T\subset B}^G(\mF)*\mG
\ra\Ind_{T\subset B}^G(\mF*\mF').
\]

\quash{
\item For any $\mF,\mF'\in D(T)_h$, there is a canonical map
\beq\label{semi group st}
\Ind_{T\subset B}^G(\mF)*\Ind_{T\subset B}^G(\mF')\ra\Ind_{T\subset B}^G(\mF*\mF').
\eeq
Moreover, assume 
$\mF,\mF'$, and 
$\mF*\mF'$ are $\rW$-equivariant local systems on $T$.
Consider the $\rW$-action (resp. diagonal $\rW$-action) on 
$\Ind_{T\subset B}^G(\mF*\mF')$ (resp. $\Ind_{T\subset B}^G(\mF)*\Ind_{T\subset B}^G(\mF')$)
in part (3).
Then the map (\ref{semi group st}) above is compatible with
those $\rW$-actions.

\item
Let $\mG\in D_G(G)$. Assume $\mF:=\on{Av}_U(\mG)$ is supported on $T=B/U$.
For any $\mF'\in D(T)$, the composition 
\beq\label{convolution}
\mG*\Ind_{T\subset B}^G(\mF')
\ra\Ind_{T\subset B}^G(\mF)*\Ind_{T\subset B}^G(\mF')
\stackrel{(\ref{semi group st})}\ra\Ind_{T\subset B}^G(\mF*\mF')
\eeq
is an isomorphism, where the first map is induced by the
canonical map \[\mG\ra\on{Av}_G\circ\on{Av}_U(\mG)\is\Ind_{T\subset B}^G(\mF)\]
in part (3) of Lemma \ref{CHHC} .}

 \end{enumerate}
\end{prop}

\subsection{Hecke categories}
Consider the left $G$ and right $T\times T$ actions on $Y=G/U\times G/U$.
To every $\xi,\xi'\in\breve T\is\breve\ft/\Lambda$ 
we denote by $M_{\xi,\xi'}$ the category of $G$-equivariant 
$D$-modules on
$G/U\times G/U$ which are $T\times T$-monodromic 
with 
generalized monodromy $(\xi,\xi')$, that is, 
$U(\ft)\otimes U(\ft)$ (acting as infinitesimal
translations along the right action of $T\times T$) acts 
locally finite with generalized eigenvalues in $(\xi,\xi')$.
Consider the quotient  $Y/T$ where $T$ acts diagonally from the right.
The group $T$ acts on $Y/T$ via the formula $t(xU,yU)\on{mod }T=
(xU,ytU)\on{mod} T$. To every $\xi\in\breve T$ we denote by
$M_\xi$ the category of $G$-equivariant $T$-monodromic $D$-modules on 
$Y/T$ with generalized monodromy $\xi$.
We write $D(M_{\xi,\xi'})$ and $D(M_\xi)$ for the corresponding 
$G$-equivariant monodromic derived categories.

The groups $B$ and $T\times T$ act on $X=G/U$ by the 
formula $b(xU)=bxb^{-1}U$, $(t,t')(xU)=txt'U$.
For any $(\xi_1,\xi_2)\in\breve T\times\breve T$ we write 
$H_{\xi_1,\xi_2}$ for the category of $U$-equivariant $T\times T$-monodromic 
$D$-modules on $X$ with generalized monodromy $(\xi_1,\xi_2)$.
For any $\xi\in\breve T$ we write 
$H_{\xi}$ 
for the category of $B$-equivariant $T$-monodromic
$D$-modules on $X$ 
with generalized monodromy $\xi$, where $B$ acts 
on $X$ by the same formula as before and  
$T$ acts on $X$ by the formula $t(xU)=txU$.
We denote by $D(H_\xi)$ (resp.
$D(H_{\xi_1,\xi_2})$)  the corresponding $B$-equivaraint (resp. $U$-equivaraint)
monodromic derived category.

Consider the embedding $i:X\ra Y, gU\ra (eU,gU)$.  
\begin{lemma}\cite{MV}\label{M and H}
\begin{enumerate}
\item The functor $i^0=i^![\dim X]:D_G(Y)\ra D_U(X)$ is an equivalence of categories with inverse 
givne by $(i^0)^{-1}:=\on{Ind}_B^G\circ i_*[\on{dim}G-\on{dim}B]$.
\item 
We have $i^0\on{HC}\is\on{Av}_U$.
\end{enumerate}
\end{lemma}

We have the convolution product
$D_G(Y)\times D_G(Y)\ra D_G(Y)$ given by $(\mF,\mF')\ra (p_{13})_*(p_{12}^*\mF\otimes p_{23}^*\mF')$.
Here $p_{ij}:G/U\times G/U\times G/U\ra Y=G/U\times G/U$
is the projection on the $(i,j)$-factors. 
The convolution product on $D_G(Y)$ restricts to a convolution product on 
$D(M_{\xi,\xi^{-1}})$. 
The equivalence $i^0:D_G(Y)\is D_U(X)$
above induces convolution products on $D_U(X)$ and $D(H_{\xi,\xi})$. 
In addition, there is an action of $D_U(X)$ on $D(X)$ by right convolution. 
The convolution operation will be denoted by $*$.

\quash{
\begin{lemma}[Proposition 9.2.1 \cite{G}]\label{central}
For any $\cM\in D_G(G)$, $\cF\in D_G(Y)$, and $\cF'\in D_U(X)$ we have 
\[\on{HC}(\cM)*\cF\is\cF*\on{HC}(\cM),\ \  \on{Av}_U(\cM)*\cF'\is\cF'*\on{Av}_U(\cM).\]
\end{lemma}
}

\quash{
\subsection{(Pro) local system $\hat\mL_\xi$}
Let $\hat\mL$ be the \emph{pro-unitpotnet} local system on $T$, that is, 
$\hat\mL=\underleftarrow{\on{lim}}\ \mL_n$ where $\mL_n$ is the local system whose fiber at the unit
element $e\in T$ is identified with 
$\on{Sym}(\ft)/\on{Sym}(\ft)_+^n$
, where the log monodromy action of $\pi_1(T)$
coincides with the 
restriction of 
natural $\on{Sym}(\ft)$-module structure to $\pi_1(T)\subset
\pi_1(T)\otimes\bC\is
\ft$.
Let $\xi\in\breve T$ and $\mL_\xi$ be the corresponding Kummer local system on $T$. We define 
the following pro-local system
\[\hat\mL_\xi=\underleftarrow{\on{lim}}(\mL_n\otimes\mL_\xi).\]

}

We will need the following lemma. 
Let $X$ be an algebraic variety with an action 
of an affine algebraic group $G$. Denote the 
action map by $a:G\times X\ra X$ .
\begin{lemma}[Lemma 2.1 \cite{BFO}]\label{action}
For any 
$\cA\in D(G)$, $\cF\in D(X)$ 
We have a canonical isomorphism 
\[R\Gamma(a_*(\mA\boxtimes\mF))\is R\Gamma(\mA)\otimes^L_{U(\fg)}R\Gamma(\cF).\]
\end{lemma}

\subsection{Character $D$-modules} 
We denote by $CS(G)$ the 
category of finitely generated $G$-equivaraint $D$-modules on $G$ such that 
the action of the center $Z\subset U(\fg)$, embedding as left invariant 
differential operators, is locally finite.
To every $\theta\in\breve T/\rW=\breve\ft/\rW_a$,
we denote by $CS_\theta(G)$ the 
category of finitely generated $G$-equivaraint $D$-modules on $G$ such that 
the action of the center $Z\subset U(\fg)$ is locally finite and has generalized eigenvalues in $\theta$.
We denote by 
$D(CS(G))$ (resp. $D(CS_\theta)$)
 the minimal triangulated full subcategory of 
$D_G(G)$ containing all objects $\mM\in D_G(G)$ such that $\mH^i(\mM)\in CS(G) $
(resp. $\mH^i(\mM)\in CS_\theta(G)$ ).
We call 
$CS(G)$ and $CS(G)_\theta$ (resp. $D(CS(G))$ and $D(CS_\theta)$) 
the category (resp. derived category)
of 
character $D$-modules on $G$ and 
character $D$-modules on $G$ 
with generalized central character $\theta$. 

We have the following:

\begin{proposition}\label{CS}
\begin{enumerate}
\item
Let $\xi\in\breve T$ be a lifting of $\theta$.
Then 
$D(CS_\theta(G))$ is generated by the image of 
$D(H_{\xi})$ (resp. $D(M_\xi)$) under the functor $\on{Av}_G:D_B(X)\ra D_G(G)$
(resp. $\on{CH}:D_G(Y/T)\ra D_G(G)$). 

\item 
Let $\mG\in CS(G)_\theta$. We have 
\[\on{HC}(\mG)\in\bigoplus_{\xi\in\breve T,\xi\ra\theta}D(M_{\xi}),\ \ (resp.\ \on{Av}_U(\mG)\in
\bigoplus_{\xi'\breve T,[\xi']=\theta}D(H_\xi).
)\]

\item 
The functors $\Ind_{T\subset B}^G$ and $\Res_{T\subset B}^G$ 
preserve the derived categories of character $D$-modules. Moreover,
the resulting functors $\Ind_{T\subset B}^G:D(CS(T))\ra D(CS(G))$,
$\Res_{T\subset B}^G:D(CS(G))\ra D(CS(T))$  
are independent of the choice of the Borel subgroup $B$ and 
t-exact with respect to the natural $t$-structures on $D(CS(G))$
and $D(CS(T))$. 

\item
Let $CS_{[T]}(G)\subset CS(G)$ be the full subcategory generated by 
the image of $\Ind_{T\subset B}^G:CS(T)\ra CS(G)$.
For any $\mG\in CS_{[T]}(G)$, the local system
$\mF=\Res_{T\subset B}^G(\mG)\in CS(T)$ carries a canonical 
$\rW$-equivariant structure, moreover, there is a canonical isomorphism 
\[
\Ind_{T\subset B}^G(\mF)^{\rW}\is\mG.
\]
Here $\Ind_{T\subset B}^G(\mF)^{\rW}$ is the $\rW$-invaraint factor of
$\Ind_{T\subset B}^G(\mF)$
for the $\rW$-action constructed in Proposition \ref{properties of ind}.

\end{enumerate}

\end{proposition}
\begin{proof}
Part (1), (2), (3) are proved in \cite{G,L}.
We now prove part (4). We first show that $\mF=\Res^G_{T\subset B}(\mG)$ is 
canonically $\rW$-equivariant.
Let $x\in N(T)$ and $w\in N(T)/T=\rW$ be its image in the Weyl group.
Denote $B_x:=\Ad_xB$.
Consider the following commutative diagram 
\[\xymatrix{T\ar[d]^{w}&B\ar[l]\ar[r]\ar[d]^{\Ad_x}&G\ar[d]^{\Ad_x}\\
T&B_x\ar[l]\ar[r]&G}
\]
where 
$w:T\ra T$ the natural action of $w\in\rW$ on $T$ and 
the horizontal arrows are the natural inclusion and projection maps.
The base change theorems and the fact that the functors 
$\Res^G_{T\subset B}$ and $\Res^G_{T\subset B_x}$ are canonical isomorphic (see part (3)) imply 
\beq\label{base change}
\Res^G_{T\subset B}(\Ad_x^*\mG)\is w^*\Res^G_{T\subset B_x}(\mG)\is w^*
\Res^G_{T\subset B}(\mG).
\eeq
Since $\mG$ is $G$-conjugation equivariant, we have a 
canonical isomorphism 
$c_x:\mG\is\Ad^*_x\mG$.
Applying $\Res^G_{T\subset B}$ to $c_x$  
and using  (\ref{base change}) we get 
\beq\label{c_w'}
\mF=
\Res^G_{T\subset B}(\mG)\is\Res^G_{T\subset B}(\Ad_x^*\mG)\is w^*\Res^G_{T\subset B}(\mG)
=w^*\mF.
\eeq
We claim that the isomorphism above depends only the image $w$
and we denote it by 
\beq\label{c_w}
c_w:\mF\is w^*\mF.
\eeq
To prove the claim it is enough to check that for $x\in T$ the 
restriction of the isomorphism (\ref{c_w'}) to $T_{rs}$ 
is equal to the identity map. By \cite{G}, the restriction 
$\mF|_{T_{rs}}$ is canonically isomorphic 
to $\mG|_{T_{rs}}$ and the map in (\ref{c_w'}) is equal to the 
restriction of $c_x$ to $T_{rs}$. 
Since 
the adjoint action $\Ad_x:G\ra G$ is trivial on $T$, 
the claim follows from 
the fact that 
any $T$-equivariant structure of a local system on $T$
is trivial.
The $G$-conjugation equivariant structure on $\mG$ implies 
$\{c_w\}_{w\in\rW}$ satisfies the required cocycle condition, hence,
the data $(\mF,\{c_w\}_{w\in\rW})$ defines a $\rW$-equivariant structure on 
$\mF=\Res^G_{T\subset B}(\mG)$.
We shall prove
$\Ind_{T\subset B}^G\cF^\rW\is\cG$.
Let $j:G_{rs}\ra G$ the natural inclusion and $c_{rs}:G_{rs}\ra T_{rs}/\rW$ 
the restriction of the Chevalley map $c:G\ra T/\rW$ to $G_{rs}$.
Note that we have 
$\mG\is j_{!*}(\mG|_{G_{rs}})$ and $\Ind^G_{T\subset B}(\mF)^\rW\is
j_{!*}(q_{rs}^*(\bar\mF))$, where 
$\bar\mF\in\on{Loc}(T_{rs}/\rW)$ is the descent of 
$\mF|_{T^{rs}}$ along the map $q_{rs}:T^{rs}\ra T^{rs}/\rW$.
So we reduce to show $\mG|_{G_{rs}}\is c_{rs}^*(\bar\mF)$
and this follows again from the fact that 
$\mG|_{T_{rs}}\is\mF|_{T_{rs}}\in\on{Loc}_{\rW}(T_{rs})$.

\end{proof}

\subsection{Local systems $\mL_\xi$, $\hat\mL_\xi$, $\mE_\xi$, and $\mE_\theta$}\label{central Loc}
Let $\xi\in\breve T$.
It defines 
an one dimensional representation $\chi_\xi$ of $\pi_1(T)$
via $\breve T\is\on{Hom}(\pi_1(T),\bG_m)$, 
and the corresponding local system on $T$, denoted by $\mL_\xi$, is 
called the Kummer local system on $T$ associates to $\xi$.
Let $\rW_\xi$ be the stabilizer of $\xi$ in $\rW$. 
Let $S:=\on{Sym}(\ft)$ and $S_+$ denote the argumentation ideal of $S$.
Define $S_n:=S/S^n_+$, $n\in\bZ_{\geq 0}$ and $S_\xi:=S/S\cdot S_+^{\rW_\xi}$.
Note the natural action of $\rW$ on $S$ induces an action of 
$\rW_\xi$ on $S_\xi$.
We view $S_n$ and $S_\xi$ as $\ft$-module by restricting the natural 
$S=\on{Sym}(\ft)$-actions on $S_n$ and $S_\xi$ to $\ft$. Obviously, the 
$\ft$-actions are nilpotent.  

Consider the following representations 
$\rho_n$
of 
$\pi_1(T)$ in the space
$S_n$, by identifying $\pi_1(T)$ with a lattice in $\ft$ and 
defining 
\beq
\rho_n(t)\cdot s=\exp(t)\cdot s
\eeq where $t\in\ft,\ s\in S_n$.
Similarly, 
we consider 
the representation 
$\rho_\xi^{uni}$ of $\rW_\xi\ltimes\pi_1(T)$ 
in the space $S_\xi$
by setting 
\beq
\rho_\xi^{uni}(w,t)\cdot u=w(\exp(t)\cdot u)
\eeq
where $(w,t)\in\rW_\xi\ltimes\pi_1(T),\ u\in S_\xi$.

Since the character 
$\chi_\xi$ is fixed by $\rW_\xi$, it extends to a 
character of $\rW_\xi\ltimes\pi_1(T)$ which we still denote by 
$\chi_\xi$. We define the following representation of $\rW_\xi\ltimes\pi_1(T)$:
\beq
\rho_{\xi}:=\rho_\xi^{uni}\otimes\chi_\xi.
\eeq
The induction 
$\Ind_{\rW_\xi\ltimes\pi_1(T)}^{\rW\ltimes\pi_1(T)}\rho_\xi$
of $\rho_\xi$ depends only on the image $\theta=[\xi]\in\breve T/\rW_\xi$ and we denote the 
resulting representation by 
\[\rho_\theta:=\Ind_{\rW_\xi\ltimes\pi_1(T)}^{\rW\ltimes\pi_1(T)}\rho_\xi.\]
 
\begin{definition}\label{Def of E_theta}
\begin{enumerate}
\item
We denote by $\mE_\theta$ the 
$\rW$-equivariant 
local systems on $T$  corresponding to the 
representation 
 $\rho_\theta$.

\item
We denote by $\mE_\xi$ and $\mE^{uni}_\xi$ the 
$\rW_\xi$-equivariant 
local systems on $T$ corresponding to the 
representation 
 $\rho_\xi$ and $\rho_\xi^{uni}$. 
\item 
We denote by $\mL^n$ (resp $\mL^n_\xi$) the 
local systems on $T$ corresponding to the 
representation 
$\rho_n$ (resp. $\rho_n\otimes\chi_\xi$). Consider the projective 
system \[\mL^1_\xi\la\mL^2_\xi\la\mL^3_\xi\cdot\cdot\]
and we define the following (pro-)local system on $T$
\beqn\label{L_xi}
\hat\mL_\xi=\underleftarrow{\on{lim}}(\mL_\xi^n).
\eeqn
\end{enumerate}
\end{definition}

\quash{
\begin{lemma}\label{E_theta}
There is a canonical $\rW$-equivariant structure on $\mE_\theta$ such that for each lifting $\xi\in\breve T$ of $\theta$ we have 
a canonical isomorphism of $\rW$-equivariant local systems 
\beq
\mE_\theta\is\on{Ind}_{\rW_\xi}^\rW\mE_\xi.
\eeq
\end{lemma}}

\subsection{Character $D$-module $\mM_\theta$}\label{CS M_theta}
Let $\theta\in\breve T/\rW$ and let $\mE_\theta$ be the 
$\rW$-equivaraint 
local system constructed in \S\ref{central Loc}. Consider $\on{Ind}_{T\subset B}^G(\mE_\theta)$ which is a $G$-equivariant
$D$-module on $G$. 
The $\rW$-equivariant structure on $\mE_\theta$ defines 
a $\rW$-action on $\on{Ind}^G_{T\subset B}(\mE_\theta)$.

\begin{definition}
We define $\mM_\theta$ to be the $\rW$-invariant factor of 
$\mathrm{Ind}^G_{T\subset B}(\mE_\theta)$
\[\mM_\theta:=\mathrm{Ind}^G_{T\subset B}(\mE_\theta)^{\rW}.\]
\end{definition}

By part (3) of Proposition \ref{CS}, each $\cM_\theta$ is a character $D$-module.
Let $\xi\in\breve T$ with image $\theta=[\xi]\in\breve T/\rW$.
Then the isomorphism 
$\mE_\theta\is\Ind_{\rW_\xi}^\rW\mE_\xi$ implies
\beq\label{M_theta}
\mM_\theta
\is(\Ind_{\rW_\xi}^\rW\on{Ind}^G_{T\subset B}(\mE_\xi))^{\rW}
\is\on{Ind}^G_{T\subset B}(\mE_\xi)^{\rW_\xi}.
\eeq

Here is a different way to construct $\mM_\theta$:
Let $G_{rs}$ be the open subset of $G$ consisting of regular 
semi-simple elements. 
The map $G_{rs}\ra T_{rs}/\rW$ induces a 
surjection 
\[\pi_1(G_{rs})\twoheadrightarrow\rW\ltimes\pi_1(T)\] 
and we 
let $\rho_{G_{rs},\theta}$ be the representation of 
$\pi_1(G_{rs})$ given by pulling back $\rho_\theta$ along the surjection above. 
We write $\mE_{G_{rs},\theta}$ for the local system on 
$G_{rs}$ corresponding to $\rho_{G_{rs},\theta}$. Then we have 
\[\mM_{\theta}\is j_{!*}\mE_{G_{rs},\theta}\]
where $j:G_{rs}\ra G$ is the open embedding.

The rest of the section is devoted to the proof of the following 
theorem:

\begin{thm}\label{Key}
We have 
$\on{Av}_U(\mM_\theta)\is\mE_\theta$. In particular, $\on{Av}_U(\mM_\theta)$
is supported on $T=B/U\subset G/U$.
\end{thm}

\subsection{Drinfeld center of Harish-Chandra bimodules}\label{center}
In this section we construct  
certain elements 
in the Drinfeld center of Harish-Chandra bimodules and identify them with
the local systems $\mE_\xi$ under the global section functor.

We first recall facts about Harish-Chandra bimodules
following  
\cite{BG,BFO}.
Let $U=U(\fg)$ be the universal enveloping algebra of $\fg$. 
Let $Z=Z(U)$ be the center of $U$. 
Consider the dot action of $\rW$ on $\breve\ft$, $w\cdot t=w(t+\rho)-\rho$, where 
$\rho$ is the sum of of positive roots. 
We have the Harish-Chandra isomorphism 
$hc:Z\is \mO(\breve\ft)^\rW$ such that for any $\lambda\in\breve\ft$
the center $Z$ acts on the Verma module associated to $\lambda$ via 
$z\ra hc(z)(\lambda)$.
For any $\lambda\in\breve\ft$  we write 
$m_\lambda$ for the corresponding maximal ideal 
and denote by $I_\lambda$ the maximal ideal  of $Z$ corresponding to 
$m_\lambda$ under the Harish-Chandra isomorphism.

Consider the extended universal enveloping algebra $\tU=U\otimes_{Z}\mO(\breve\ft)$.
where $Z$ acts on $\mO(\breve\ft)$ via the Harish-Chandra isomorphism.
We denote by $\tU_\lambda=\tU/\tU m_\lambda$
and  
$\tU_{\hat\lambda}=\underleftarrow{\on{lim}} (\tU/\tU m_\lambda^n)$.

We denote by $\mathcal{HC}_{\hat\lambda}$ the category of 
finitely generated Harish-Chandra bimodules over 
$\tU_{\hat\lambda}$, that is, finitely generated continuous $\tU_{\hat\lambda}$-bimodules such that the 
diagonal action of $\fg$ is locally finite.
We denote by $D(\mathcal{HC}_{\hat\lambda})$ the corresponding derived category.
The tensor product 
$\mM\otimes\mM':=\mM\otimes_{\tU_{\hat\lambda}}\mM'$, 
$\mM,\mM'\in HC_{\hat\lambda}$
(resp. $\mM\otimes^L\mM':=\mM\otimes^L_{\tU_{\hat\lambda}}\mM'$, $\mM,\mM'\in D(\mathcal{HC}_{\hat\lambda})$) defines a monoidal structure on $\mathcal{HC}_{\hat\lambda}$
(resp. $D(\mathcal{HC}_{\hat\lambda})$).

\begin{proposition}[\cite{BG,BFO}]\label{monoidal}
Let $\xi\in\breve T=\breve\ft/\Lambda$ and $\lambda\in\breve\ft$ be a dominate regular lifting of $\xi$. 
The functor 
\[R\Gamma^{\hat\lambda,\widehat{-\lambda-2\rho}}:(D(M_{\xi,\xi^{-1}}),*)\is (D(\mathcal{HC}_{\hat\lambda}),\otimes^L)\] is an 
equivalence of monoidal categories.
\end{proposition}

Let $\xi\in\breve T$ and $\lambda\in\breve\ft$ a dominant regular lifting of 
$\xi$.
By Proposition \ref{monoidal} 
we have an equivalence of monoidal categories
\beq\label{equ M}
\mathrm M:(D(H_{\xi,\xi}),*)\is (D(M_{\xi,\xi^{-1}}),*)\stackrel{R\Gamma^{\hat\lambda,\widehat{-\lambda-2\rho}}}\is (D(\mathcal{HC}_{\hat\lambda}),\otimes^L).
\eeq
Consider the following 
full 
abelian subcategory 
$H_{\xi,\xi}^t:=\mathrm M^{-1}(\mathcal{HC}_{\hat\lambda})\subset D(H_{\xi,\xi})$.
For any $\mM,\mM'\in H_{\xi,\xi}^t$
we define 
\[\mM*^t\mM':=\mathrm M^{-1}(\mathrm M(\mM)\otimes\mathrm M(\mM'))\in H_{\xi,\xi}^t.\]
One can check that $*^t$ defines a monoidal structure on $M_{\xi,\xi}^t$ and the functor $\mathrm M$ induces an
equivalence of abelian monoidal categories 
\beq\label{mon equ ab setting}
\mathrm M^t:(H_{\xi,\xi}^t,*^t)\is(\mathcal{HC}_{\hat\lambda},\otimes)
\eeq

\quash{Let $\lambda\in\breve\ft$. 
The map $\breve\ft\ra\breve\ft, x\ra x-\lambda$ induces a map of algebras 
$a_\lambda:S\ra S, f\ra (x\ra f(x+\lambda))$ (recall $S=\on{Sym}(\ft)\is\mO(\breve\ft)$).
Consider the following action of $S$ on $S_\xi$, $s\cdot m=a_\lambda(s)m$,
$s\in S,m\in S_\xi$ and denote the resulting $S$-module by $S_\xi^{\lambda}$.}
For each $\lambda\in\breve\ft$ we 
consider the following $S$-module $S_\xi^\lambda$: we have $S_\xi^\lambda=S_\xi$ as vector spaces 
and the $S$-module structure is given by 
$s\cdot m=a_\lambda(s)m$, $s\in S,m\in S_\xi^\lambda$, where 
$a_\lambda:S\ra S, f\ra (x\ra f(x+\lambda))$.
We define 
\beq\label{central element}
\cZ_\lambda:=\tU_{\hat\lambda}\otimes_SS_\xi^\lambda\in\mathcal{HC}_{\hat\lambda}
\eeq
where $S$ acts on $\tU_{\hat\lambda}$ via the 
the map $
p_\lambda:S\is Z(\tU)\ra Z(\tU_{\hat\lambda})$. 

We denote by 
$Z(\mathcal{HC}_{\hat\lambda},\otimes)$
 (resp. $Z(H^t_{\xi,\xi},*^t)$) the Drinfeld center of the monoidal category
$Z(\mathcal{HC}_{\hat\lambda},\otimes)$ (resp. $Z(H^t_{\xi,\xi},*^t)$). Recall an element in $Z(\mathcal{HC}_{\hat\lambda},\otimes)$ consists of an element 
$\mM\in \mathcal{HC}_{\hat\lambda}$ together with family of compatible isomorphisms $b_\mF:\mM\otimes\mF\is\mM\otimes\mF$
for $\mF\in \mathcal{HC}_{\hat\lambda}$. 
\quash{
We have a natural forgetful map 
$For:Z_\xi\ra M_{\xi,\xi}^t$ sending $(\mM,b_\mF)_{\mF\in H_{\xi,\xi}^t}\ra\mM$.}

\begin{prop}\label{E=Z}
Let 
$\lambda\in\breve\ft$ be a dominant regular weight 
and $\xi\in\breve T$ be its image.
\begin{enumerate}
\item To every $\mM\in \mathcal{HC}_{\hat\lambda}$ there is a canonical isomorphism 
\[b_\mM:\mathcal Z_\lambda\otimes\mM\is\mM\otimes\cZ_\lambda\] such that 
the data $(\cZ_\lambda,b_\mM)_{\mM\in \mathcal{HC}_{\hat\lambda}}$ defines 
an element in the Drinfeld center $Z(\mathcal{HC}_{\hat\lambda},\otimes)$.

\item  
We have $\mathrm M(\cE_\xi)\is\cZ_\lambda$.
\end{enumerate}
\end{prop}
\begin{proof}
Proof of (1).
Consider the map 
$m_{\lambda}:S\otimes S\stackrel{a_\lambda\otimes a_{-\lambda-2\rho}}\ra S\otimes S
\stackrel{p_\lambda\otimes p_{-\lambda-2\rho}}
\ra Z(\tU_{\hat\lambda}\otimes\tU_{\widehat{-\lambda-2\rho}})$.
To every $\mM\in \mathcal{HC}_{\hat\lambda}$, the map above 
defines an action of $S\otimes S$ on $\mM$ and the result in \cite{S} implies 
that this action factors through $S\otimes S\ra S\otimes_{S^{\rW_\xi}}S$.
Therefore, for every $\mM\in \mathcal{HC}_{\hat\lambda}$, we have a 
canonical
isomorphism $b_\mM:\cZ_\lambda\otimes\mM
\is S/S\cdot S^{\rW_\xi}_+\otimes_S\mM\is\mM\otimes_SS/S\cdot S^{\rW_\xi}_+\is\mM\otimes\cZ_\lambda$. One can check 
that those isomorphisms satisfy the required compatibility conditions and 
the data $(\cZ_\lambda,b_\mM)$
defines an element in $Z(\mathcal{HC}_{\hat\lambda},\otimes)$.

Proof of (2).
Let $\tilde\mE_\xi\in M_{\xi}$
be the image of $\mE_\xi$ under the equivalence 
$(i^{0})^{-1}:H_{\xi}\is M_{\xi}$ in Lemma \ref{M and H}.
Then by definition we have $\mathrm M(\mE_\xi)\is R\Gamma^{\hat\lambda,\widehat{-\lambda-2\rho}}(\pi^*\tilde\mE_\xi)$, 
where $\pi:Y\ra Y/T$.
Consider the map \[a:T\times (G/U\times G/U)/T\ra (G/U\times G/U)/T,\ (t,gU,g'U)\ra 
(gt^{-1}U,g'U).\] 
Then it follows from the definition of $(i^0)^{-1}$ that we have $\tilde\mE_\xi=a_*(\mE_\theta\boxtimes\Delta_*\mO_{G/B})$, here $\Delta:G/B\ra (G/U\times G/U)/T$ is the embedding $gB\ra (gU,gU)\on{mod}\ T$.
Note that $R\Gamma(\Delta_*\mO_{G/B}\otimes p_2^*\omega_{G/B})\is\tU$
($p_2$ is the right projection map $(G/U\times G/U)/T\ra G/B$) hence by Lemma \ref{action} we get
\[
R\Gamma^{\hat\lambda,\widehat{-\lambda-2\rho}}(\pi^*\tilde\mE_\xi)\is
R\Gamma^{\hat\lambda}(\tilde\mE_\xi\otimes p_2^*\omega_{G/B})=R\Gamma^{\hat\lambda}(a_*(\mE_\xi\boxtimes(\Delta_*\mO_{G/B}\otimes p_2^*\omega_{G/B}))\is
\tU\otimes_{\on{Sym}(\ft)} \Gamma^{\hat\lambda}(\mE_\xi).\]
Since $\Gamma^{\hat\lambda}(\mE_\xi)\is S_\xi^\lambda$,
part (2) follows.

\end{proof}

\begin{corollary}\label{central structure}
\begin{enumerate}
\item We have $\mE_\xi\in H^t_{\xi,\xi}$.
\item  
To every $\mM\in H^t_{\xi,\xi}$ there is a canonical isomorphism 
\[b_\mM:\mE_\xi*^t\mM\is\mM*^t\mE_\xi\] such that 
the data $(\mE_\xi,b_\mM)_{\mM\in H^t_{\xi,\xi}}$ defines 
an element in the Drinfeld center $Z(H^t_{\xi,\xi},*^t)$.
\end{enumerate}
\end{corollary}
\begin{proof}
This follows immediately from above proposition and (\ref{mon equ ab setting}).
\end{proof}

\quash{
1)
The functor $\on{Av}_U^\xi$ maps $CS_\theta$ to $M_{\xi,\xi}^t$. 
We denote the resulting functor by $\on{Av}$
\\
2) 
There is a canonical equivalence 
\[{\on{av}}_U^\xi:CS_{\theta}\is Z(H^t_{\xi,\xi},*^t)\]
such that the composition 
$CS_{\theta}\is Z(H^t_{\xi,\xi},*^t)\ra M_{\xi,\xi}^t$
is isomorphic to $\on{Av}_U^\xi$.
}

\subsection{Proof of Theorem \ref{Key}}
Recall the notion of translation functor $\theta_\lambda^\mu:\mathcal{HC}_{\hat\lambda}\ra\mathcal{HC}_{\hat\mu}$
where $\mu\in\lambda+\Lambda$. 
In \cite{BFO}, they proved the following:

(1) There is a lifting  $\theta_\lambda^\mu:Z(\mathcal{HC}_{\hat\lambda},\otimes)\ra Z(\mathcal{HC}_{\hat\mu},\otimes)$ such that 
the functor
$\mathbf F:Z(\mathcal{HC}_{\hat\lambda},\otimes)\ra Z(\mathcal{HC}_{\hat\xi},\otimes),\ 
L\ra\bigoplus_{\mu\in(\lambda+\Lambda)/\rW_\xi}\theta_\lambda^\mu(L)$
define an equivalence of braided monoidal categories.

(2) For any  $\mM\in\cM_G(G)$ the global section  $\Gamma(\mM)$
is naturally a 
Harish-Chandra bimodule, with a canonical central structure and 
the resulting functor  $\Gamma:\cM_G(G)\ra Z(\mathcal{HC},\otimes)$
is 
an equivalence of abelian categories. Moreover, the equivalence above 
restricts to an equivalance $CS_\theta\is Z(\mathcal{HC}_{\hat\xi},\otimes)$
and the composed equivalence 
\[CS_\theta\is Z(\mathcal{HC}_{\hat\xi},\otimes)
\stackrel{\mathbf F^{-1}}\is Z(\mathcal{HC}_{\hat\lambda},\otimes)\]
is isomorphic to $R\Gamma^{\hat\lambda,\widehat{-\lambda-2\rho}}\circ\pi^0\on{HC}$.
Here $\pi:Y\ra Y/T$ is the projection map.

Let $\cZ_\lambda\in Z(\mathcal{HC}_{\hat\lambda},\otimes)$, $\tilde\mE_\xi\in M_{\xi}$ 
be as 
in Proposition \ref{E=Z}.  
Define 
$\tilde\mE_\theta\is\bigoplus_{\xi\in\breve T,\xi\ra\theta}\tilde\mE_\xi\in\cM_G(Y)$.
By the discussion above there exists a character $D$-module $\cM_\theta'\in CS_\theta$ such that 
\[R\Gamma^{\hat\lambda,\widehat{-\lambda-2\rho}}\circ\pi^0\on{HC}(\cM_\theta')=\cZ_\lambda.\]
Hence by Proposition \ref{E=Z}, we have   
\[R\Gamma^{\hat\lambda,\widehat{-\lambda-2\rho}}\circ\pi^0\mathrm{HC}(\mM_\theta')\is
R\Gamma^{\hat\lambda,\widehat{-\lambda-2\rho}}(\pi^0\tilde\mE_\xi)\]
for any regular dominant $\lambda\in\breve\ft$
mapping to $\xi$.
Since $\pi^0\on{HC}:D(CS_\theta)\ra\bigoplus_{\xi\in\breve T,\xi\ra\theta} D(M_{\xi,\xi^{-1}})$
and 
$R\Gamma^{\hat\lambda,\widehat{-\lambda-2\rho}}:
D(M_{\xi,\xi^{-1}})\is \mathcal{HC}_{\hat\lambda}$ is an equivalence of category for regular dominant $\lambda$,
this implies $\mathrm{HC}(\mM_\theta')\is\tilde\mE_\theta$.
Applying the equivalence $i^0:D(M_\xi)\is D(H_\xi)$ on both sides and using 
Lemma \ref{M and H}, we get 
\beq\label{vanishing}
\on{Av}_U(\mM_\theta')\is i^0(\mathrm{HC}(\mM_\theta'))\is i^0\tilde\mE_\theta\is\mE_\theta.
\eeq
The isomorphism above implies 
$\mE_\theta\is\on{Av}_U(\mM_\theta')\is\Res_{T\subset B}^G(\mM_\theta')$, hence 
by part $(4)$ of Proposition \ref{CS}, there 
is canonical $\rW$-equivaraint structure on $\mE_\theta$
such that $\mM_\theta'\is\Ind_{T\subset B}^G(\mE_\theta)^\rW$
.  In the lemma below 
we will show that this $\rW$-equivaraint structure on $\mE_\theta$
coincides with the one in Definition \ref{Def of E_theta}, therefore we have 
$\mM_\theta\is\mM_\theta'\is\Ind_{T\subset B}^G(\mE_\theta)^\rW$
and the theorem follows from (\ref{vanishing}).

\begin{lemma}
The $\rW$-equivaraint structure on 
$\mE_\theta\is\Res_{T\subset B}^G(\mM_\theta')$
constructed 
in Proposition \ref{CS}
coincides with the one in Definition \ref{Def of E_theta}. 
\end{lemma}
\begin{proof}
We give a proof in the case when $\theta=[e]\in\breve T/\rW$. The proof 
for the general cases are similar. 
By construction we have $\cZ_0=\tU_{\hat 0}\otimes_S(S/S^\rW_+)$,
$\mM'_\theta\is\bigoplus_{\mu\in \Lambda/\rW}\theta_{0}^\mu(\cZ_0)$, 
and $\mE_\theta=\mE_e=\mE_e^{unip}$.
Let $x\in N(T)$ and $w\in\rW$ its image in the Weyl group. 
Let \[c_x:\mM'_\theta\is\Ad_x^*\mM'_\theta\] 
be the isomorphism coming from 
the $G$-conjugation equivariant structure on $\mM'_\theta$
and 
\[c_w:\mE_\theta\is\Res_{T\subset B}^G\mM_\theta'\is\Res_{T\subset B}^G\Ad_x^*\mM_\theta'\is w^*\mE_\theta\]
the induced map in (\ref{c_w}).
By Lemma \ref{action} we have 
\beq\label{section}
R\Gamma^{\hat 0}(\on{Av}_U(\mM_\theta'))\is (U(\fg)/U(\fg)\fn)\otimes_SS/S^\rW_+,\ \ R\Gamma^{\hat 0}(\on{Av}_U(\Ad_x^*\mM_\theta'))\is (U(\fg)/U(\fg)\fn_x)\otimes_SS/S^\rW_+.
\eeq
Here in the right isomorphism $\fn_x=\Ad_x\fn$ and the $S$-module structure on $S/S^\rW_+$
is given by $s\cdot m=\Ad_x(s)m$, $s\in S, m\in S/S^\rW_+$.
Since the map $c_x$ comes from the 
adjoint $G$-action on $\cZ_0$, under the isomorphisms in (\ref{section}),
the map \[
R\Gamma^{\hat 0}(\on{Av}_U(c_x)):
R\Gamma^{\hat 0}(\on{Av}_U(\mM_\theta'))\is R\Gamma^{\hat 0}(\on{Av}_U(\Ad_x^*\mM_\theta'))\]
becomes  
\[(U(\fg)/U(\fg)\fn)\otimes_SS/S^\rW_+\ra(U(\fg)/U(\fg)\fn_x)\otimes_SS/S^\rW_+,\ v\otimes s\ra \Ad_x(v)\otimes s.\]
This implies the map 
\[
a_w:S/S^\rW_+\is R\Gamma^{\hat 0}(\mE_\theta)\stackrel{R\Gamma^{\hat 0}(c_w)}\ra R\Gamma^{\hat 0}(w^*\mE_\theta)\is S/S^\rW_+
\]
induced by $c_w$
is equal to the natural $\rW$-action map $s\ra \Ad_xs=w(s)$. 
Notice that the assignment $w\ra a_w$ is the representation of $\rW$
on $S/S^\rW_+\is\mE_\theta|_e$, coming from the 
$\rW$-equivariant structure $\{c_w\}_{w\in\rW}$ on $\mE_\theta$.
Since the $\rW$-equivariant structure on $\mE_\theta$ constructed in 
Definition \ref{Def of E_theta} uses 
the same representation of $\rW$ on $S/S^\rW_+$,
the lemma follows.

\end{proof}

\section{Gamma $D$-modules}
In this section we recall the definition of gamma $D$-modules on a reductive group due to D.Kazhdan and A.Braverman. 
We compute the 
convolution of 
gamma $D$-modules with the $D$-modules 
$\cE_\xi$ and $\cM_\theta$
introduced in \S\ref{CS}
(see Proposition \ref{key prop} and 
Theorem \ref{conv with M_theta}). Those computations play important roles 
in \S\ref{KB conj} for the proof of 
the acyclicity of Gamma $D$-modules 
and 
in \S\ref{NL FT} for the proof of exactness properties of non-linear Fourier transforms.

\subsection{Gamma $D$-modules on $T$}\label{Gamma on T}
For any $c\in\bC^\times$ we 
consider the corresponding exponential $D$-module 
$\bC[x]e^{cx}:=\{fe^{cx}|f\in\bC[x]\}$
on $\bG_a=\Spec\bC[x]$ with generator $e^{cx}$ 
and $\partial_x(fe^{cx})=(f'+c)e^{cx},\ f\in\bC[x]$.
For each nontrivial cocharacter $\lambda:\bG_m\ra T$ we define 
$\Psi(\lambda,c):=\lambda_*(j^*\bC[x]e^{cx})$ where $j:\bG_m\ra\bG_a$ is the natural inclusion.
Note that since $\lambda$ is finite we have \[\Psi(\lambda,c)=\lambda_*(j^*\bC[x]e^{cx})\is\lambda_!(j^*\bC[x]e^{cx}).\]
Recall the convolution product $
*$ (resp. $\star$) on $D(T)_{hol}$
\[\mF*\mF'=m_*(\mF\boxtimes\mF'),\  (\text{resp.}\  
\mF\star\mF'=m_!(\mF\boxtimes\mF'))\] 
Here $m:T\times T\ra T, (x,y)\ra xy$ is the multiplication map. 
For every collection of possibly repeated nontrivial cocharacter $\underline\lambda=(\lambda_1,...,\lambda_r)$
we define 
\[\Psi_{\underline\lambda,c}:=\Psi(\lambda_1,c)*\cdot\cdot\cdot*\Psi(\lambda_r,c)\]
\[\Psi^!_{\underline\lambda,c}:=\Psi(\lambda_1,c)\star\cdot\cdot\cdot\star\Psi(\lambda_r,c).\]
Following \cite{BK1}, we call
$\Psi_{\ul,c}$ and $\Psi_{\ul,c}^!$ gamma $D$-modules on $T$.

Let $\pr_{\underline\lambda}:\bG_m^r\ra T, (x_1,...,x_r)\ra\prod_{i=1}^r\lambda_i(x_i)$
and $\tr:\bG_m^r\ra\bG_a, (x_1,...,x_r)\ra\sum x_i$.
Then using base changes one can show that 
$\Psi_{\underline\lambda,c}=\pr_{\underline\lambda,*}\tr^*(\bC[x]e^{cx})$
and $\Psi^!_{\underline\lambda,c}=\pr_{\underline\lambda,!}\tr^*(\bC[x]e^{cx})$.
\quash{
Following \cite{K}, we will call $\Psi_{\underline\lambda}$ and $\Psi^!_{\underline\lambda}$ hypergeometric 
$D$-modules. }

\begin{definition}\label{sigma positive}
Let $\sigma:T\ra\bG_m$ be a character. A co-character $\lambda$ is called $\sigma$-positive 
if $\sigma\circ\lambda:\bG_m\ra\bG_m$ has the form $t\ra t^n$ for some positive integer $n$.
\end{definition}
We have the following properties of gamma $D$-modules:
\begin{proposition}\label{Psi and L}
Let $\ul=\{\lambda_1,...,\lambda_r\}$ be a collection of  $\sigma$-positive
co-characters.
\\
1) Then the natural map $\Psi^!_{\underline\lambda,c}\ra
\Psi_{\underline\lambda,c}$ is an isomorphism. Moreover, $\Psi_{\underline\lambda,c}$ is a local system on 
the image $T_{\underline\lambda}:={\on{Im}(\pr_{\underline\lambda}})\subset T$.
\\
2)
We have $\mathbb D(\Psi_{\ul,c})\is\Psi_{\ul,-c}$.
\\
3) Let $\mL$ be a Kummer local system on $T$. We have \[H^i_{\text{dR}}(\Psi_{\underline\lambda,c}\otimes\mL)=0\] for $i\neq 0$ and 
$\dim H^0_{\text{dR}}(\Psi_{\underline\lambda,c}\otimes\mL)=1$. Moreover, we have 
a canonical isomorphism
\[\Psi_{\underline\lambda,c}*\mL\is H^0_{\text{dR}}(\Psi_{\underline\lambda,c}\otimes\mL^{-1})\otimes\mL.\]
\\
4) 
Consider the functor 
of right convolution with $\Psi_{\ul,c}\in D_U(X)$:
\[(-)*\Psi_{\ul,c}:D(X)\ra D(X).\]
The functor above 
preserves the $T$-monodromic subcategory $D(X)_{mon}\subset D(X)$, where $T$ acts on $X=G/U$ from the right.
\quash{
, and the resulting functor 
$(-)*\Psi_\ul:D(X)_{mon}\ra D(X)_{mon}$ is exact with respect to the 
natural $t$-structure on $D(X)_{mon}$.}
\end{proposition}
\begin{proof}
Part 1), 2), 3) are proved in \cite[Theorem 4.2, Theorem 4.8 ]{BK1}.
Part 4) follows from part 3).
\end{proof}
Let $\underline\lambda=(\lambda_1,..,\lambda_r)$ be a collection of 
$\sigma$-positive cocharacters. 
Recall that the Weyl group $\rW$ acts naturally 
on $\mathbb X_\bullet(T)$ and we assume that the set $\{\lambda_i\}_{i=1,...,r}$ are invariant under 
this action. Following \cite{BK} (see also \cite{CN}), we shall construct a $\rW$-equivariant structure on $\Psi_{\ul}$.
Let $(\lambda_{i_1},...,\lambda_{i_k})$ be the distinct cocharacters appearing in $\ul$ and 
$m_l$ be the multiplicity of $\lambda_{i_l}\in\ul$. Let $A_i=\{\lambda_l|\lambda_l=\lambda_{i_l}\}$.
Then we have $\{\lambda_1,...,\lambda_r\}=A_1\sqcup...\sqcup A_k$.
The symmetric group on $r$-letters $\mathrm S_r$ acts naturally on $\{\lambda_1,...,\lambda_r\}$
and we define  $\rS_\ul=\{\sigma\in\rS_r|\sigma(A_i)=A_i\}$. 
There is
a canonical isomorphism 
\[\rS_\ul\is\rS_{m_1}\times\cdot\cdot\cdot\times\rS_{m_k}.\]
Define $\rS'_\ul=\{\eta\in\rS_r|\text{such that}\ \eta(A_i)=A_{\tau(i)}\ \text{for a}\ \tau\in\rS_k\}$.
We have a natural map $\pi_k:\rS_\ul'\ra\rS_k$ sending $\eta$ to $\tau$. The 
kernel of $\pi_k$ is isomorphic to $\rS_\ul$ and 
its image, denote by $\rS_{k,\ul}$, 
consists of $\tau\in\rS_k$ such that $m_i=m_{\tau(i)}$. 
In other words, there is a short exact sequence 
\[0\ra\rS_\ul\ra\rS_\ul'\stackrel{\pi_k}\ra\rS_{k,\ul}\ra 0.\] 

Notice that the Weyl group $\rW$ acts on $\{\lambda_{i_1},...,\lambda_{i_k}\}$
and the induced map $\rW\ra\rS_k$ has image $\rS_{k,\ul}$.
So we have a map $\rho:\rW\ra\rS_{k,\ul}$. Pulling back the short exact sequence 
above along $\rho$, we get an extension of $W'$ of $W$ by $\rS_\ul$
\[0\ra\rS_\ul\ra W'\ra W\ra 0\] 
where an element in $w'\in W'$ consists of pair $(w,\eta)\in\rW\times\rS_\ul'$ such
that $\rho(w)=\pi_k(\eta)\in\rS_{k,\ul}$.

The group $\rW'$ acts on $\bG_m^r$ (resp. $T$) via the composition of 
the action of $\rS_r$ (reps. $\rW$) with the natural projection $\rW'\ra\rS_\ul'\subset\rS_r$ 
(resp. $\rW'\ra\rW$) and the map $\pr_\ul:\bG_m^r\ra T$ and $\tr:\bG_m^r\ra\bG_a$ is $W'$-equivaraint
where $\rW'$ acts trivially on $\bG_a$. 

Since $\Psi_{\ul,c}\is(\pr_\ul)_*\tr^*(\bC[x]e^{cx})$, the discussion above 
implies for each $w'=(w,\eta)\in\rW'$ there is an isomorphism 
\[i'_{w'}:\Psi_{\ul,c}\is w^*\Psi_{\ul,c}.\]
We define 
\beq\label{i_w'}
i_{w'}=\on{sign}_r(\eta)\on{sign}_{\rW}(w)i'_{w'}:\Psi_{\ul,c}\is w^*\Psi_{\ul,c}
\eeq
where $\on{sign}_r$ and $\on{sign}_{\rW}$ are the sign characters of $\rS_r$
and $\rW$. 
According to \cite{BK1}, the isomorphism $i_{w'}$ depends only on $w$. Denote the resting isomorphism by $i_w$,
then the data $(\Psi_{\ul,c},\{i_w\}_{w\in\rW})$ defines a $\rW$-equivariant structure on $\Psi_{\ul,c}$.

\subsection{Mellin transform}\label{MT}
Let $x_i\in\Lambda$ be a basis and consider the 
regular function $\mO(T)\is\bC[x_i^{\pm1}]$ and the algebra 
of differential operators 
$\Gamma(\mD_T)\is\bC[x_i^{\pm1}]\langle v_i\rangle/\{v_ix_j=x_j(\delta_{ij}+v_i)\}$
where $v_i=x_i\partial_{\xi}\in\ft$ are a basis for the $T$-invariant 
vector fields.
The Mellin transform functor 
\[\mathfrak M:\cM(T)\ra\bC[v_i]\on{-mod}^\Lambda,\  \cN\ra\Gamma(\cN),\]
defined by considering $\Gamma(\mD_T)$ as algebra of difference 
operators $\bC[v_i]\langle x_i^{\pm1}\rangle/\{v_ix_j=x_j(\delta_{ij}+v_i)\}$,
is 
an equivalence of abelian categories between 
$D$-modules on $T$ and  
$\Lambda$-equivariant $\bC[v_i]$-modules.
Consider the derived category 
of $\Lambda$-equivariant 
$\bC[v_i]$-modules $D(\bC[v_i]\on{-mod}^\Lambda)$
with the monoidal 
structure given by the derived tensor product over 
$\bC[v_i]$. 
We have 
$\mathfrak M(\cM*\cN)\is\mathfrak M(\cM)\otimes^L_{\bC[v_i]}\mathfrak M(\cN)$.  

Let $\rW'\subset\rW$ be a subgroup.
Consider the category 
$\cM_{\rW'}(T)$ (resp. 
$\bC[v_i]\on{-mod}^{\rW'\ltimes\Lambda}$)
of $\rW'$-equivariant $D$-module on $T$ (resp. 
$\rW'\ltimes\Lambda$-equivariant $\bC[v_i]$-modules).
Then the functor $\mathfrak M$ extends naturally to an equivalence of 
categories
\[\mathfrak M_{\rW'}:\cM_{\rW'}(T)\ra
\bC[v_i]\on{-mod}^{{\rW'}\ltimes\Lambda}.\]

Note that we have a canonical equivalence 
$\bC[v_i]\on{-mod}^{\Lambda}\is\on{QCoh}(\ft^*)^{\Lambda}$ (resp. 
$\bC[v_i]\on{-mod}^{\rW'\ltimes\Lambda}\is\on{QCoh}(\breve\ft)^{\rW'\ltimes\Lambda}$)
where $\on{QCoh}(\breve\ft)^{\Lambda}$ (resp. $\on{QCoh}(\breve\ft)^{\rW\ltimes\Lambda}$) is the category of 
$\Lambda$-equivariant (resp. $\rW'\ltimes\Lambda$-equivariant) quasi-coherent sheaves on $\breve\ft$.
We write 
$\underline{\mathfrak M}:\cM(T)\ra\on{QCoh}(\ft^*)^{\Lambda}$
(resp. $\underline{\mathfrak M}_{\rW'}:\cM_{\rW'}(T)\ra\on{QCoh}(\ft^*)^{\rW'\ltimes\Lambda}$) for the composition of 
$\mathfrak M$ (resp. $\mathfrak M_{\rW'}$) with the equivalence above.

Let $(\cN,c_w:\cN\is w^*\cN)\in\cM_{\rW'}(T)$. Then
the $\rW'$-equivariant structure on $\mathfrak M_{\rW'}(\cN)$
is determined by the following $\rW'$-action on $\mathfrak M(\cN)$:
\beq\label{W'-action}
a_w:\mathfrak M(\cN)=\Gamma(\cN)\stackrel{\Gamma(c_w)}\is\Gamma(w^*\cN)\is\Gamma(\cN)=
\mathfrak M(\cN),\ w\in\rW'.
\eeq

\subsection{Mellin transform of $\Psi_{\ul,c}$}
We describe the Mellin transform of
$\Psi_{\ul,c}$. 
Recall \[\Psi_{\ul,c}\is(\pr_\ul)_*\on{tr}^*(\bC[x]e^{cx})\is
(\bC[x_1^{\pm1},...,x_r^{\pm1}]e^{\sum cx_i}\otimes_{\mO_{\bG_m^r}}\omega_{\bG_m^r})\otimes
_{\mD_{\bG_m^r}}\mD_{\bG_m^r\ra T}\otimes_{\mO_T}\omega_T^{-1}. \]
Fix a nowhere vanishing $\bG_m^r$-invariant (resp. $T$-invariant) section 
$r_1\in\Gamma(\omega_{\bG_{m}^r})$ (resp. $r_2\in\Gamma(\omega_T)$). Then the trivialization $\cO_{\bG_m^r}\is\omega_{\bG_m^r}$
(resp. $\mO_{T}\is\omega_T$)
induced by $r_1$ (resp. $r_2$) defines an isomorphism 
\[t_\ul:\mathfrak M(\Psi_{\ul,c})\is\bC[x_1^{\pm1},...,x_r^{\pm1}]e^{\sum cx_i}\otimes_{\bC[v_1,...,v_r]}S
\]
where $v_i=x_i\partial_{x_i}$
and the $\bC[v_1,...,v_r]$-module structure on $S=\on{Sym}(\ft)$ 
is given by 
$v_i\cdot s=d\lambda_i(v_i)s$, $s\in S$, here $d\lambda_i:\bC[v_i]\ra S$
is the differential of the cocharacter $\lambda_i$.
Recall the $\rW$-action on $\mathfrak M(\Psi_{\ul,c})$:
\[a_w:\mathfrak M(\Psi_{\ul,c})=\Gamma(\Psi_{\ul,c})\stackrel{\Gamma(i_{w})}\ra
\Gamma(w^*\Psi_{\ul,c})\is\Gamma(\Psi_{\ul,c})=\mathfrak M(\Psi_{\ul,c}),\ w\in\rW.\]
Let $w'=(w,\eta)\in\rW'$. We have 
$\eta^*(r_1)=\on{sign}(\eta)r_1$ (resp. $w^*r_2=\on{sign}(w)r_2$) and 
the construction of $i_{w'}$ (see (\ref{i_w'}))
implies the following description of 
$a_w$: 
the following diagram commutes
\beq\label{description of i_w}
\xymatrix{\mathfrak M(\Psi_{\ul,c})\ar[d]^{a_w}\ar[r]^{t_\ul\ \ \ \ \ \ \ \ \ \ \ \ \ \ \ \ }&\bC[x_1^{\pm1},...,x_r^{\pm1}]e^{\sum cx_i}\otimes_{\bC[v_1,...,v_r]}S\ar[d]^{b_{w'}}
\\\mathfrak M(\Psi_{\ul,c})\ar[r]^{t_\ul\ \ \ \ \ \ \ \ \ \ \ \ \ \ \ \ }&\bC[x_1^{\pm1},...,x_r^{\pm1}]e^{\sum cx_i}
\otimes_{\bC[v_1,...,v_r]}S},
\eeq
where  
$b_{w'}:\bC[x_1^{\pm1},...,x_r^{\pm1}]e^{\sum cx_i}\otimes_{\bC[v_1,...,v_r]}S\ra\bC[x_1^{\pm1},...,x_r^{\pm1}]e^{\sum cx_i}\otimes_{\bC[v_1,...,v_r]}S 
$, $f\otimes s\ra \eta(f)\otimes w(s)$.

\subsection{Mellin transform of 
$\mE_\xi$ and $\mL_\xi^n$}\label{MT of E_xi}
We describe the Mellin transforms of $\mE_\xi$ and $\mL^n_\xi$.
To every $\mu\in\breve\ft$ let $l_\mu:\breve\ft\ra\breve\ft, v\ra v-\mu$.
Recall the $S$-module $S_\xi=S/S^{\rW_\xi}_+$, $S_n=S/S_+^n$
introduced in \S\ref{central Loc}.
Let $\cS_\xi:=\mO_{\breve\ft}\otimes_SS_\xi$, 
$\cS_n
:=\mO_{\breve\ft}\otimes_SS_n$ be quasi-coherent sheaves on 
$\breve\ft$ corresponding to $S_\xi=S/S^{\rW_\xi}_+$ and $S_n=S/S_+^n$. We define $\cS_\xi^\mu:=l_\mu^*\cS_\xi$, $\cS_n^{\mu}:=l_\mu^*\cS_n$
and write
$S_\xi^\mu=\Gamma(\cS_\xi^\mu)$, $S_n^{\mu}=\Gamma(\cS_\xi^\mu)$.
Note that the $S$-module $S_\xi^\mu$ here agrees with the one 
in (\ref{central element}).
Clearly we have 
$l^*_\lambda\cS_\xi^\mu\is\cS_\xi^{\mu+\lambda}$
(resp.
$l^*_\lambda\cS_n^{\mu}\is\cS_n^{\mu+\lambda}$)
 for $\lambda\in\ft^*$.
Consider the projection map $\pi_\xi:\breve\ft\ra\breve\ft/\rW_\xi$.
We have
$\cS_\xi\is\pi_\xi^*\delta$, where $\delta$ the the skyscraper sheaf supported at 
$0\in\ft^*/\rW_\xi$, and 
the equality $\pi_\xi\circ l_{w(\mu)}\circ w=\pi_\xi\circ l_\mu$
defines an isomorphism 
$\cS_\xi^\mu\is w^*\cS_\xi^{w(\mu)}$. 

The isomorphisms 
$l^*_\lambda\cS_n^{\mu}\is\cS_n^{\mu+\lambda}$ for $\lambda\in\Lambda$
defines a 
$\Lambda$-equivariant structure on 
$\bigoplus_{\mu\in\ft^*,[\mu]=\xi}\cS_n^{\xi,\mu}$ and we have 
\[\underline{{\mathfrak M}}_{}(\mL_\xi^n)\is\bigoplus_{\mu\in\breve\ft,[\mu]=\xi}\cS_n^{\mu}\in
\on{QCoh}(\breve\ft)^{\Lambda},\ \ 
(\on{resp}.\  {\mathfrak M}_{}(\mL_\xi^n)\is\bigoplus_{\mu\in\breve\ft,[\mu]=\xi}S_n^{\mu}
\in S\on{-mod}^{\Lambda}).\]
Similarly,
the isomorphisms 
$\cS_\xi^{\mu+\lambda}\is
l^*_\lambda\cS_\xi^\mu, \lambda\in\Lambda$
and $\cS_\xi^\mu\is w^*\cS_\xi^{w(\mu)}$ define a $\rW_\xi\ltimes\Lambda$-equivariant structure on 
$\bigoplus_{\mu\in\ft^*,[\mu]=\xi}\cS_\xi^\mu$ and we have  
\[\underline{{\mathfrak M}}_{\rW_\xi}(\mE_\xi)\is\bigoplus_{\mu\in\breve\ft,[\mu]=\xi}\cS_\xi^\mu\in
\on{QCoh}(\breve\ft)^{\rW_\xi\ltimes\Lambda},\ \
(\on{resp.}\  {\mathfrak M}_{\rW_\xi}(\mE_\xi)\is\bigoplus_{\mu\in\breve\ft,[\mu]=\xi}S_\xi^\mu\in
S\on{-mod}^{\rW_\xi\ltimes\Lambda}).\]
Let 
\[
u_w^\mu:S_\xi^\mu=\Gamma(\cS_\xi^\mu)\ra\Gamma(w^*\cS_\xi^{w(\mu)})\is\Gamma(\cS_\xi^{w(\mu)})=S_\xi^{w(\mu)}
\]
be the map induced by the isomorphism $\cS_\xi^\mu\is w^*\cS_\xi^{w(\mu)}$.
Then the $\rW_\xi$-action on $\mathfrak M(\mE_\xi)$, defined in (\ref{W'-action}), 
decomposes as 
\beq\label{u_w}
u_w=\oplus u_w^\mu:\mathfrak M(\mE_\xi)=\bigoplus_{\mu\in\breve\ft,[\mu]=\xi}
S_\xi^\mu\ra 
\bigoplus_{\mu\in\breve\ft,[\mu]=\xi}
S_\xi^{w(\mu)}
=\mathfrak M(\mE_\xi),\ w\in\rW_\xi.
\eeq

\subsection{}

We have the following key proposition whose proof will be given in section \ref{proof of key prop}.
\begin{proposition}\label{key prop}
There is an isomorphism 
\[\Psi_{\ul,c}*\mE_\xi\is\mE_\xi\]
of $\rW_\xi$-equivaraint local systems on $T$. 
\end{proposition}

\quash{
\begin{proof}
Since $\hat\mL_\xi*\mE_\xi\is\mE_\xi$
it suffices to show that \[\Psi_\ul*\hat\mL_\xi\is V_{\ul,\xi}\otimes\hat\mL_\xi\] as 
$\rW_\xi$-equivaraint (pro-)local systems on $T$

\end{proof}}

\

\subsection{Gamma $D$-modules on $G$}\label{cons of gamma d mod}
We preserve the setup in \S\ref{Gamma on T}. Let $\ul=(\lambda_1,...,\lambda_r)$
be a collection of $\rW$-invariant $\sigma$-positive co-characters. 
By Proposition \ref{properties of ind}, the $\rW$-equivariant structure on $\Psi_{\ul,c}$ defines a 
$\rW$-action on $\Ind_{T\subset B}^G(\Psi_{\ul,c})$.

\begin{definition}
The gamma $D$-module attached to $\ul$ is the $\rW$-invariant factor of 
the $D$-module $\Ind_{T\subset B}^G(\Psi_{\ul,c})$
\[\Psi_{G,\ul,c}:=\Ind_{T\subset B}^G(\Psi_{\ul,c})^{\rW}.\]
\end{definition}

\quash{
\begin{remark}
In the original definition of gamma $D$-modules \cite{BK} (see also \cite{CN}), the authors use the $\rW$-equivariant structure 
$i_w':=\on{sign}(w)i_w:w^*\Psi_\ul\is\Psi_\ul$ on $\Psi_\ul$ and define 
the gamma $D$-module as 
the $\rW$-invariant factor 
$\Psi_{G,\ul}=\Ind_{T\subset B}^G(\Psi_\ul)^\rW$
for the $\rW$-action on $\Ind_{T\subset B}^G(\Psi_\ul)$ induced by the 
$\rW$-equivariant structure $(\Psi_\ul,i_w')$. Clearly, those two definitions 
of gamma $D$-modules 
are equivalent.
\end{remark}
}

Giving a representation $\rho:\breve G\ra GL(V_\rho)$, its restriction 
$\rho|_{\breve T}$ is diagonalizable, i.e. there exist 
a collection of co-characters $\ul_\rho=\{\lambda_1,...,\lambda_r\}
$ such that $V_{\rho}=\bigoplus V_{\lambda_i}$ where 
$\breve T$ acts on $V_{\lambda_i}$ by the $\lambda_i\in
\Hom(\bG_m,T)\is\Hom(\breve T,\bG_m)$. Note that 
$\ul_\rho$ is automatically $\rW$-stable. Assume 
each $\lambda_i\in\ul_\rho$ is $\sigma$-positive 
and $\pr_{\ul_\rho}$ is onto, then the gamma $D$-module (or rather,
the corresponding gamma sheaf)
$\Psi_{G,\rho,c}:=\Psi_{G,\ul_\rho,c}$ attached to $\ul_\rho$
is the one studied in \cite{BK,BK1}.

The following property of gamma $D$-module follows from Proposition \ref{Psi and L}:
\beq\label{dual of Psi}
\mathbb D(\Psi_{G,\ul,c})\is\Psi_{G,\ul,-c}.
\eeq

Recall the character $D$-module 
$\cM_\theta$ in \S\ref{CS M_theta}.
\begin{thm}\label{conv with M_theta}
There is an isomorphism 
\[\Psi_{G,\ul,c}*\mM_\theta\is \mM_\theta.\]

\end{thm}
\begin{proof}
Fix a lifting of $\xi\in\breve T$ of $\theta$. Then we have 
 $\on{Av}_U(\mM_\theta)\is\mE_\theta\is\Ind_{\rW_\xi}^\rW\mE_\xi$, which is supported on $T=B/U$. Thus
by Proposition \ref{properties of ind} and Proposition \ref{key prop} we have 
\[
\Ind_{T\subset B}^G(\Psi_{\ul,c})*\mM_\theta\is\Ind_{T\subset B}^G(\Psi_{\ul,c}*\mE_{\theta})
\is\Ind^\rW_{\rW_\xi}(\Ind_{T\subset B}^G(\Psi_{\ul,c}*\mE_\xi))\is
\Ind^\rW_{\rW_\xi}(\Ind_{T\subset B}^G(\mE_\xi))
.\]
Now taking $\rW$-invariant on both sides of the isomorphism above 
and using (\ref{M_theta}),
we arrive 
\[\Psi_{G,\ul,c}*\mM_\theta\is
(\Ind_{T\subset B}^G(\Psi_{\ul,c})*\mM_\theta)^\rW\is
(\Ind_{T\subset B}^G(\Psi_{\ul,c}*\mE_\xi))^{\rW_\xi}
\is\Ind_{T\subset B}^G(\mE_\xi)^{\rW_\xi}\is\mM_\theta.\]

\end{proof}

\subsection{Proof of Proposition \ref{key prop}}\label{proof of key prop}
\quash{It suffices to construct an isomorphism 
$\mathfrak M(\Psi_\ul*\mE_\xi)\is\mathfrak M(\mE_\xi)$ 
which is compatible with the $\rW_\xi$-actions coming from the 
$\rW_\xi$-equivariant structures on $\Psi_{\ul,c}*\mE_\xi$ and $\mE_\xi$ (see section \ref{}).}
We shall construct an isomorphism 
$\Psi_{\ul,c}*\mE_\xi\is\mE_\xi$. 
For this we will first 
construct an isomorphism
$\Psi(\lambda,c)*\mE_\xi\is\mE_\xi$ for $\lambda\in\breve T$. 
For simplicity, we will write 
$\Psi_\ul$ (resp. $\Psi(\lambda)$) for $\Psi_{\ul,c}$ (resp. $\Psi(\lambda,c)$).
By \S\ref{MT} and \S\ref{MT of E_xi}, we have 
$\mathfrak M(\mE_\xi)=\bigoplus_{\mu\in\ft^*,[\mu]=\xi} S_\xi^\mu$ 
and  
\beq\label{Eq1}
\mathfrak M(\Psi(\lambda)*\mE_\xi)\is\mathfrak M(\Psi(\lambda))\otimes_S\mathfrak M(\mE_\xi
)\is\bigoplus_{\mu\in\ft^*,[\mu]=\xi}\bC[x^{\pm 1}]e^{cx}\otimes_{\bC[v]}S_\xi^\mu,
\eeq
here $x$ is a coordinate of $\bG_m$, $v=x\partial_x$, and $\bC[v]$ acts on
$S_\xi^\mu$ via the map 
$d\lambda:\bC[v]\ra S$.
Write $\lambda(\mu)=a_{\lambda,\mu}+n_{\lambda,\mu}$, with $a_{\lambda,\mu}\in[0,1)$
, $n_{\lambda,\mu}\in\bZ$, 
and consider the free $\bC[v]$-submodule 
\[E_{\lambda,\mu}:=\bC[v]\cdot x^{n_{\lambda,\mu}}e^{cx}\subset\bC[x^{\pm1}]e^{cx}\]
generated by $x^{n_{\lambda,\mu}}e^{cx}$.
From the relation $v\cdot x^ne^{cx}=(nx^n+cx^{n+1})e^{cx}$,
we deduce that 
$\bC[x^{\pm1}]e^{cx}/E_{\lambda,\mu}$, as a quasi-coherent sheaf 
on $\Spec{\bC[v]}\is\bC$,
is supported away from 
$\lambda(\mu)$. Since $S_\xi^\mu$ is supported on $\lambda(\mu)$, we deduce 
that $(\bC[x^{\pm1}]e^{cx}/E_{\lambda,\mu})\otimes_{\bC[v]}S_\xi^\mu=0$ and 
\beq\label{Eq2}
\bC[x^{\pm1}]e^{cx}\otimes_{\bC[v]}S_\xi^\mu\is E_{\lambda,\mu}\otimes_{\bC[v]}S_\xi^\mu
\is S_\xi^\mu.
\eeq
Combining (\ref{Eq1}) and (\ref{Eq2}) we get 
\[
\mathfrak M(\Psi(\lambda)*\mE_\xi)\is
\bigoplus_{\mu\in\ft^*, [\mu]=\xi}E_{\lambda,\mu}\otimes_{\bC[v]}S_\xi^\mu\is
\bigoplus_{\mu\in\ft^*, [\mu]=\xi} S_\xi^\mu\is\mathfrak M(\mE_\xi)
\]
and this gives 
\beq\label{Psi and E_xi}
\Psi(\lambda)*\mE_\xi\is\mE_\xi.
\eeq
The isomorphism above defines an isomorphism
\beq\label{kappa}
\kappa:\Psi_\ul*\mE_\xi\is\Psi(\lambda_1)*\cdot\cdot\cdot*\Psi(\lambda_r)*\mE_\xi\is\mE_\xi.
\eeq

We shall show that $\kappa$ is compatible with the $\rW_\xi$-equivariant structures on both sides. 
According to (\ref{W'-action}),
it suffices to show that the map 
$\mathfrak M(\kappa):\mathfrak M(\Psi_\ul*\mE_\xi)\is\mathfrak M(\mE_\xi)$
is compatible with the $\rW_\xi$-actions on both sides. Denote $x_{\ul,\mu}:=\prod_{i=1}^rx_i^{n_{\lambda_i,\mu}}$ and 
consider the 
free $\bC[v_1,...,v_r]$-submodule 
\[E_{\ul,\mu}:=\bC[v_1,...,v_r]\cdot x_{\ul,\mu}e^{\sum cx_i}\subset
\bC[x_1^{\pm 1},...,x_r^{\pm1}]e^{\sum cx_i}\]
generated by $x_{\ul,\mu}e^{\sum cx_i}$.
It follows from (\ref{Eq2}) 
that 
\[
\mathfrak M(\Psi_{\ul}*\mE_\xi)\is
\bigoplus_{\mu\in\ft^*,[\mu]=\xi}
\bC[x_1^{\pm 1},...,x_r^{\pm1}]e^{\sum cx_i}
\otimes_{\bC[v_1,...,v_r]}S_\xi^\mu\is
\bigoplus_{\mu\in\ft^*,[\mu]=\xi}
E_{\ul,\mu}\otimes_{\bC[v_1,...,v_r]}S_\xi^\mu
.\]
Moreover, under the isomorphism above,
the map $\mathfrak M(\kappa)$
becomes 
\beq\label{Eq4}
\bigoplus_{\mu\in\ft^*,[\mu]=\xi}
E_{\ul,\mu}\otimes_{\bC[v_1,...,v_r]}S_\xi^\mu\is
\bigoplus_{\mu\in\ft^*,[\mu]=\xi}
S_\xi^\mu,\ \ x_{\ul,\mu}\otimes s\ra s.
\eeq

We now describe the $\rW_\xi$-action  
on $\mathfrak M(\Psi_{\ul}*\mE_\xi)$.
Let $w'=(w,\eta)\in\rW'$ with $w\in\rW_\xi$. Consider the map
\[a_\eta:\bC[x_1^{\pm 1},...,x_r^{\pm1}]e^{\sum cx_i}\ra
\bC[x_1^{\pm 1},...,x_r^{\pm1}]e^{\sum cx_i},\  fe^{\sum cx_i}\ra\eta(f)
e^{\sum cx_i}.\]
Since 
\[a_\eta(x_{\ul,\mu}e^{\sum cx_i})=(\prod_{i=1}^r x_{\eta(i)}^{n_{\lambda_i,\mu}})e^{\sum cx_i}=
(\prod_{i=1}^rx_i^{n_{w^{-1}(\lambda),\mu}})e^{\sum cx_i}=
(\prod_{i=1}^rx_i^{n_{\lambda,w(\mu)}})e^{\sum cx_i}=x_{\ul,w(\mu)}e^{\sum cx_i}\]
the map $a_\eta$ restricts to a map    
$a_\eta:E_{\ul,\mu}\ra E_{\ul,w(\mu)}\subset\bC[x_1^{\pm 1},...,x_r^{\pm1}]e^{\sum cx_i}$. 
Consider the $\rW_\xi$-action on $\mathfrak M(\Psi_\ul*\mE_\xi)$:
$a_w:\mathfrak M(\Psi_\ul*\mE_\xi)\ra\mathfrak M(\Psi_\ul*\mE_\xi),\ w\in\rW_\xi$. It follows from the 
description for the $\rW_\xi$-action on $\mathfrak M(\Psi_\ul)$ in (\ref{description of i_w}) 
that we have the following commutative diagram 
\beq\label{Eq5}
\xymatrix{\bigoplus_{\mu\in\ft^*,[\mu]=\xi}E_{\ul,\mu}\otimes_{\bC[v_1,...,v_r]}S_\xi^\mu\ar[d]
^{\oplus_{\mu}a_\eta\otimes u^\mu_w}
\ar[r]&\mathfrak M(\Psi_\ul*\mE_\xi)\ar[d]^{a_w}
\\
\bigoplus_{\mu\in\ft^*,[\mu]=\xi}E_{\ul,w(\mu)}\otimes_{\bC[v_1,...,v_r]}S_\xi^{w(\mu)}\ar[r]&\mathfrak M(\Psi_\ul*\mE_\xi)}
\eeq
where
\beq\label{Eq6}
u_w=\oplus u_w^\mu:\mathfrak M(\mE_\xi)=\bigoplus_{\mu\in\ft^*,[\mu]=\xi}S_\xi^\mu\ra 
\bigoplus_{\mu\in\ft^*,[\mu]=\xi}S_\xi^{w(\mu)}=\mathfrak M(\mE_\xi),\ \ w\in\rW_\xi
\eeq 
is map in
(\ref{u_w}) describing the $\rW_\xi$-action on $\mathfrak M(\mE_\xi)$.
Since the map in (\ref{Eq4}) satisfies the following commutative diagram 
\[
\xymatrix{
\bigoplus_{\mu\in\ft^*,[\mu]=\xi} E_{\ul,\mu}\otimes_{\bC[v_1,...,v_r]}S_\xi^\mu\ar[d]^{\oplus a_\eta\otimes u^\mu_w}\ar[r]^{\ }&\bigoplus_{\mu\in\ft^*,[\mu]=\xi} S_{\xi}^\mu
\ar[d]^{\oplus u^\mu_w}
\\
\bigoplus_{\mu\in\ft^*,[\mu]=\xi} E_{\ul,w(\mu)}\otimes_{\bC[v_1,...,v_r]}S_\xi^{w(\mu)}\ar[r]^{
}&\bigoplus_{\mu\in\ft^*,[\mu]=\xi} S_\xi^{w(\mu)}},
\]
we deduce from (\ref{Eq5}) and (\ref{Eq6}) that 
$\mathfrak M(\kappa):\mathfrak M(\Psi_\ul*\mE_\xi)\is\mathfrak M(\mE_\xi)$ is compatible with the 
$\rW_\xi$-action on both sides.
This finishes the proof of the proposition.

\subsection{}
Recall the local systems 
$\mL_\xi^n$ in \S\ref{central Loc}. We have 
$\mathfrak M(\mL_{\xi}^n)=\bigoplus_{\mu\in\breve\ft,[\mu]=\xi} S_{n}^{\mu}$ 
(see  \S\ref{MT of E_xi}).
Using the relation $v\cdot x^ne^{cx}=(nx^n+cx^{n+1})e^{cx}$ and the fact that 
$S_{n}^{\mu}$, viewing as $\bC[v]$-module via $d\lambda:\bC[v]\ra S$, is supported on 
$\lambda(\mu)$, the same argument as in the proof of (\ref{kappa}) gives: 
\begin{lemma}\label{Psi_c}
There exists a projective system of isomorphisms 
\[\Psi_{\ul,c}*\mL^n_\xi\is\mL_\xi^n.\]
\end{lemma}

\quash{
and the $u_w$ is the map in (\ref{}) coming from the 
$\rW_\xi$-equivariant structure on $\mE_\xi$.
On the other hand, it follows from the description in (\ref{}) that 
$\eta\otimes w:E_{\ul,\mu}\otimes_{\bC[v_1,...,v_r]}S_\xi^\mu
\ra E_{\ul,w(\mu)}\otimes_{\bC[v_1,...,v_r]}S_\xi^{w(\mu)}$
is equal to the 
$\Gamma(\Psi_\ul*\mE_\xi)\is\Gamma(\mE_\xi)$, compatible with 
the $\rW_\xi$-acitons on both sides coming from the $\rW_\xi$-equivaraint structures.
To this end, we observe that 
$\Gamma(\mE_\xi)=\bigoplus_{\mu\in\ft^*,[\mu]=\xi} S_\xi^\mu$ and 
\[\Gamma(\Psi_\ul*\mE_\xi)\is\Gamma(\Psi_\ul)\otimes_S\Gamma(\mE_\xi
)\is\bigoplus_{\mu\in\ft^*,[\mu]=\xi} \Gamma(\Psi_\ul)\otimes_SS_\xi^\mu.\]
}

\section{Kazhdan-Braverman conjecture}\label{KB conj}
In \cite[Conjecture 9.2]{BK} A.Braverman and D.Kazhdan 
conjectured the following vanishing property of 
gamma $D$-module:
\begin{conjecture}
$\on{Av}_{U!}(\Psi_{G,\ul,c})$ is supported on $T=B/U\subset G/U$.
Here 
$\on{Av}_{U!}:D_G(G)_{hol}\ra D_B(G/U)_{hol}$ is the 
shriek averaging functor in (\ref{Av!}).
\end{conjecture}

Since $\mathbb D(\on{Av}_{U!}(\Psi_{G,\ul,c}))\is\on{Av}_{U}(\mathbb D(\Psi_{G,\ul,c}))\is
\on{Av}_{U}(\Psi_{G,\ul,-c})$ (see (\ref{dual of Psi})), the 
conjecture above is equivalent to the following: 
\begin{thm}\label{main thm}
$\on{Av}_U(\Psi_{G,\ul,c})\in D_B(G/U)$
is supported on $T=B/U\subset G/U$.
\end{thm}
\begin{proof}
It suffices to show that the natural map $r:\on{Res}_{T\subset B}^G(\Psi_{G,\ul,c})\ra\on{Av}_U(\Psi_{G,\ul,c})$
is an isomorphism. 
We claim that the convolution 
\beq\label{convolution}
\on{Res}_{T\subset B}^G(\Psi_{G,\ul,c})*\mE_\theta\ra\on{Av}_U(\Psi_{G,\ul,c})*\mE_\theta
\eeq
of $r$ with $\mE_\theta$ is an isomorphism for all $\theta\in\breve T/\rW$.
For this, it is enough to show that 
$\on{Av}_U(\Psi_{G,\ul,c})*\mE_\theta$ is supported on $T$ and this follows from 
Theorem \ref{Key} and Theorem \ref{conv with M_theta}. Indeed, we have  
$\on{Av}_U(\Psi_{G,\ul,c})*\mE_\theta\stackrel{\on{Thm \ref{Key}}}\is\on{Av}_U(\Psi_{G,\ul,c})*\on{Av}_U(\cM_\theta)\is
\on{Av}_U(\Psi_{G,\ul,c}*\mM_\theta)\stackrel{\on{Thm} \ref{conv with M_theta}}\is\on{Av}_U(\mM_\theta)\stackrel{\on{Thm \ref{Key}}}\is\mE_\theta$.
Note
$\mE_\theta\is\Ind_{\rW_\xi}^\rW\mE_\xi$ and 
(\ref{convolution}) implies that $\on{cone}(r)$,
the cone  of 
$r$, satisfies 
$\on{cone}(r)*\mE_{\xi}=0$ for all $\xi\in\breve T$. 
Since $\mE_{\xi}$ is a local system on $T$ with generalized monodromy $\xi\in\breve T$,
Lemma \ref{vanishing 2} and Lemma \ref{vanishing 1} below imply $\on{cone}(r)=0$. The theorem follows.
\end{proof}

\begin{corollary}\label{Av of Psi}
We have $\on{Av}_U(\Psi_{G,\ul,c})\is\Psi_{\ul,c}$.
\end{corollary}
\begin{proof}
Indeed, by \cite[Theorem 6.6]{BK1} we have 
$\on{Res}_{T\subset B}^G(\Psi_{G,\ul,c})\is\Psi_{\ul,c}$. Thus 
the theorem above implies
$\on{Av}_U(\Psi_{G,\ul,c})\is\on{Res}_{T\subset B}^G(\Psi_{G,\ul,c})\is\Psi_{\ul,c}$.
\end{proof}

\subsection{Vanishing lemmas}\label{vanishing lemmas}
Let $X$ be a smooth variety with a free $T$ action $a:T\times X\ra X$.
For $\mL\in D(T)$ and $\mF\in D(X)$ we define 
$\mL*\mF:=a_*(\mL\boxtimes\mF)\in D(X)$.

\begin{lemma}\label{vanishing 2}
Let $\mL$ be a local system on $T$ with generalized monodromy $\xi\in\breve T$, that is, 
$\mL\otimes\mL_\xi$ is an unipotent local system.  
Let $\mF\in D(X)_{hol}$ and 
assume $\mL*\mF=0$. Then we have $\mL_\xi*\mF=0$.
\end{lemma}
\begin{proof}
There is a filtration $0=\mL^{(0)}\subset\mL^{(1)}\subset\cdot\cdot\cdot\subset\mL^{(k)}=\mL$
such that \[0\ra\mL^{(i-1)}\ra\mL^{(i)}\ra\mL^{(i)}/\mL^{(i-1)}\is\mL_\xi\ra 0.\]
Assume $\mL_\xi*\mF\neq 0$ and let $m$ be the smallest number such that $\mathcal H^{\geq m}(\mL_\xi*\mF)=0$.
An induction argument, using above short exact sequence, shows that $\mH^{\geq m}(\mL^{(i)}*\mF)=0$
for $i=1,...,k$. 
Now since $\mL*\mF=0$, the 
distinguished triangle 
\[\mL^{(k-1)}*\mF\ra\mL*\mF\ra\mL_\xi*\mF\ra\mL^{(k-1)}*\mF[1]\]
implies 
\[\mL_\xi*\mF\is\mL^{(k-1)}*\mF[1].\]
Therefore we have $\mH^{m-1}(\mL_\xi*\mF)\is\mH^{m-1}(\mL^{(k-1)}*\mF[1])=\mH^{m}(\mL^{(k-1)}*\mF)=0$
which contradicts to the fact that $m$ is the smallest number such that $\mH^{\geq m}(\mL_\xi*\mF)=0$.
We are done.
\end{proof}

\begin{lemma}\label{vanishing 1}
Let $\mF\in D(X)_{hol}$. If  
$\mL_\xi*\mF=0$ for all $\xi\in\breve T$, then $\mF=0$.
\end{lemma}
\begin{proof}
Since $T$ acts freely on $X$ we have an
embedding $o_x:T\ra X, t\ra t\cdot x$. Moreover, by base change, we have
\[R\Gamma_{\on{dR}}(T,\mL_{\xi^{-1}}\otimes^! o_x^!\mF)\is i_x^!(\mL_\xi*\mF)=0\]
for all $\xi\in\breve T$. Here $i_x:x\ra X$ is the natural inclusion map.
By
\cite[Proposition 3.4.5]{GL} it implies $o_x^!\mF=0$ for all $x$.\footnote{In \cite{GL} they proved the vanishing result in the setting of 
$\ell$-adic sheaves, but the same proof works for the setting of 
$D$-modules.} The lemma follows.


\end{proof}

\section{Non-linear Fourier transforms}\label{NL FT}
In this section we fix a $c\in\bC^\times$ and write 
$\Psi_{G,\ul}=\Psi_{G,\ul,c},\Psi_{\ul}=\Psi_{\ul,c}$.
Following 
Braverman-Kazhdan, we consider the functor of convolution with 
gamma $D$-module:
\[\mathrm F_{G,\ul}:=(-)*\Psi_{G,\ul}:D(G)_{hol}\ra D(G)_{hol},\ \mF\ra\mF*\Psi_{G,\ul}.\]
The result in \cite{BK1} (see, for example, \cite[Theorem 5.1]{BK1})
suggests that the functor $\mathrm F_{G,\ul}$ can be thought as a version of \emph{non-linear Fourier transform} on the 
derived category of holonomic $D$-modules. 

The following property of 
$\mathrm F_{G,\ul}$
 follows from Proposition \ref{properties of ind}, Proposition \ref{Psi and L}, 
 and Theorem \ref{main thm}:

\begin{thm}($\mathrm F_{G,\ul}$ commutes with induction 
functors)\label{commutes with ind}
For every $\mF\in D(T)_{hol}$ we have 
\[\mathrm F_{G,\ul}(\Ind_{T\subset B}^G(\mF))\is\Ind_{T\subset B}^G(\mathrm F_{T,\ul}(\mF)).\] 
Here $\mathrm F_{T,\ul}(\mF):=\mF*\Psi_\ul$.
In particular, for any Kummer local system
$\mL_\xi$ on $T$ we have 
\[\mathrm F_{G,\ul}(\Ind_{T\subset B}^G(\mL_\xi))\is V_{\ul,\xi}\otimes\Ind_{T\subset B}^G(\mL_\xi).\]
Here $V_{\ul,\xi}:=H^0_{\text{dR}}(\Psi_{\ul}\otimes\mL^{-1}_\xi)$.
\end{thm}

We have the following conjecture:
\begin{conjecture}[see Conjecture 6.8 in 
\cite{BK1}]
$\mathrm F_{G,\ul}$ is an exact functor.

\end{conjecture}

We shall prove 
a weaker statement which says that 
$\mathrm F_{G,\ul}$ is 
exact on the category of admissible 
$D$-modules.
We first recall the definition of admissible modules following 
\cite{G}.

\begin{definition}\label{admissible}
A holonomic $D$-module $\mF$ on $G$ is called admissible if 
the action of the center $Z(\rU(\fg))$ of $\rU(\fg)$, viewing as invariant differential operators, is locally finite.
We denote by $\mA(G)$ the abelian category of admissible $D$-modules on $G$ and $D(\mA(G))$
be the corresponding derived category.
\end{definition}

\begin{remark}
We do not require 
admissible $D$-modules to be $G$-equivariant with respect to the 
conjugation action. So the definition of admissible modules here is more general than the one in \cite{G}.
\end{remark}

We have the following characterization of admissible modules:
a $\cF\in\cM(G)_{hol}$ is admissible 
if and only if  $\on{HC}(\cF)\in D(Y/T)$ is monodromic 
with respect to the right $T$-action, or equivalently,
$\on{Av}_U(\cF)\in D(X)$ is monodromic with respect to the right 
$T$-action.

To every $
\theta\in\breve T/\rW$,
let $\mA(G)_{\theta}$ be the full subcategory of 
$\mA(G)$ consisting of holonomic $D$-modules on $G$ such that  
$Z(\rU(\fg))$ acts locally finitely 
with generalized eigenvalues in $\theta$.
The category $\mA(G)$ decomposes as 
\[\mA(G)=\bigoplus_{\theta\in\breve T/\rW}\mA(G)_{\theta}.\]
\quash{
Moreover, according to \cite{G}, we have $\cM\in D(\mA(G))_{\theta}$
if and only if $\on{HC}(\cM)\in\bigoplus_{\xi\in\breve T,[\xi]=\theta} D(Y/T)_{\xi}$
, or equivalently, 
$\on{Av}_U(\cM)\in\bigoplus_{\xi\in\breve T,[\xi]=\theta} D(X)_{\xi}$.}

\begin{thm}\label{exact of FT}
The functor $\mathrm F_{G,\ul}:D(G)_{hol}\ra D(G)_{hol},\ \mF\ra\mF*\Psi_{G,\ul}$ preserves the subcategory 
$D(\mA(G))$ and the resulting functor 
\[\mathrm F_{G,\ul}:D(\mA(G))\ra D(\mA(G))\]
is exact with respect to the natural $t$-structure.
That is, we have $\mathrm F_{G,\ul}(\mF)\in\mA(G)$ for $\mF\in\mA(G)$.

\end{thm}
\begin{proof}
We show that $\mathrm F_{G,\ul}$ preserves 
$D(\mA(G))$. Using the characterization of admissible modules above 
we have to show that $\on{Av}_U(\mathrm F_{G,\ul}(\mM))$
is monodromic for $\cM\in D(\cA(G))$.
Since $\on{Av}_U(\Psi_{G,\ul})\is\Psi_\ul$ by Corollary \ref{Av of Psi}, we have 
\[\on{Av}_U(\mathrm F_{G,\ul}(\mM))\is\on{Av}_U(\cM)*\on{Av}_U(\Psi_{G,\ul})
\is\on{Av}_U(\cM)*\Psi_\ul\]
which is 
$T$-monodromic by Proposition \ref{Psi and L}. The claim follows.

We show that $\mathrm F_{G,\ul}$ is exact on $\cA(G)$.
Let $\cO_Y$ (resp. $\cO_X$) be 
pre-image of the open $G$-orbit (resp. $B$-orbit)
in $\mB\times\mB$ (resp. $\mB$) under the projection map 
$Y\ra\mB\times\mB$ (resp. $X\ra\mB$).
The quotient $G\backslash\cO_Y$ (resp. $U\backslash\cO_X$) is 
a torsor over $T$, choosing a trivialization of the torsor, we get a map
$p_Y:\cO_{Y}\ra T$ (resp. $p_X:\cO_{X}\ra T$).
We denote by 
$j_Y:\cO_{Y}\ra Y$
(resp. $j_X:\cO_{X}\ra X$) the natural embedding. 
Consider the following
pro-object in $M_{\xi,w_0(\xi^{-1})}$
(resp. $H_{\xi^{-1},w_0(\xi^{-1})}$):
\[\mathrm I_{Y}:=j_{Y!}(p_Y^0\hat\mL_{w_0(\xi^{-1})}):=\underleftarrow{\on{lim}}\ j_{Y!}
(p_Y^0\hat\mL_{w_0(\xi^{-1})}^n)
\ \ (\text{resp.}\ \mathrm I_{X}:=j_{X!}(p_X^0\hat\mL_{w_0(\xi^{-1})}):=\underleftarrow{\on{lim}}
\ j_{X!}
(p_X^0\hat\mL_{w_0(\xi^{-1})}^n)).\]
Recall the notion of intertwining functor (see \cite{BB,BG})
\[(-)*\mathrm I_Y:D(Y)_{\xi,\xi^{-1}}\ra D(Y)_{\xi,w_0(\xi^{-1})}\ \ (\text{resp.}\ 
(-)*\mathrm I_X:D(X)_{\xi^{-1},\xi^{-1}}\ra D(X)_{\xi^{-1},w_0(\xi^{-1})}).\]

According to \cite[Corollary 3.4]{BFO}, the assignment $\cM\ra\on{HC}(\cM)*\mathrm I_{Y}:=
\underleftarrow{\on{lim}}\on{HC}(\cM)* j_{Y!}(p_Y^0\mL_{\xi}^n),\  \cM\in D(G)_{hol}$
restricts to a functor 
 \[\on{HC}(-)*\mathrm I_{Y}:D(\mA(G)_{[\xi]})\ra D(Y)_{\xi,w_0(\xi^{-1})}\]
which is $t$-exact and conservative\footnote{The definition of intertwining functor here is different from that of \cite{BFO}, though one can show that 
the two definition are equivalent. In \cite{BFO}, the intertwining functor is described 
as a shriek convolution with certain $G$-equivariant $D$-module 
on $Y$.}.
So to prove the exactness of $\mathrm F_{G,\ul}$ it suffices to show that
\[\on{HC}(
\mathrm F_{G,\ul}(\cM))*\mathrm I_{Y}
\in\cM(Y)_{\xi,w_0(\xi^{-1})}\] for all 
$\mM\in\mA(G)_{[\xi]}$.  
We claim that there is an isomorphism of pro-objects 
\beq\label{claim}
\on{HC}(\Psi_{G,\ul})*\mathrm I_Y\is\mathrm I_Y.
\eeq
Thus
\[\on{HC}(
\mathrm F_{G,\ul}(\cM))*\mathrm I_{Y}
\is\on{HC}(\cM*\Psi_{G,\ul})*\mathrm I_Y\is\on{HC}(\cM)*\on{HC}(\Psi_{G,\ul})*\mathrm I_Y\is\on{HC}(\cM)*\mathrm I_Y\]
which is in $\cM(Y)_{\xi,w_0(\xi^{-1})}$ by the exactness of the 
functor $\on{HC}(-)*\mathrm I_{Y}$. We are done.

Proof of the claim. Applying the equivalence 
$i^0:D_G(Y)\is D_U(X)$ to (\ref{claim}) and using 
$i^0(\on{HC}(\Psi_{G,\ul}))\is\on{Av}_U(\Psi_{G,\ul})\is\Psi_\ul$
, $i^0(\mathrm I_Y)\is\mathrm I_X$, we reduce to show that there is an isomorphism of 
pro-objects
$\Psi_\ul*\mathrm I_X\is\mathrm I_X$. 
Note that we have $\hat\mL_{\xi^{-1}}*\mathrm I_X\is\mathrm I_X$\footnote{Indeed, 
it follows from the fact that the functor 
$\hat\mL_{\xi}*(-):D_U(X)\ra \text{pro}(D_U(X))$ (here $\text{pro}(D_U(X))$
is the category of pro-objects in $D_U(X)$), when restricts to 
the subcategory $D(H_{\xi,\xi'})$ consisting of $T\times T$-monodromic 
complexes with generalized monodromy $(\xi,\xi')$, is isomorphic to the identity functor.
 },
hence by
Lemma \ref{Psi_c} 
\[\Psi_\ul*\mathrm I_X\is\Psi_\ul*\hat\mL_{\xi^{-1}}*\mathrm I_X
\is\hat\mL_{\xi^{-1}}*\mathrm I_X\is\mathrm I_X.\]
The claim follows.


\end{proof}

\end{document}